\documentclass[10pt]{amsart}
\usepackage{eurosym}
\usepackage{graphicx,amscd,color,amsmath,amsfonts,amssymb,geometry}

\def \ent {\mathrm{ent}}
\newtheorem{theorem}{Theorem}[section]

\newtheorem{corollary}[theorem]{Corollary}
\newtheorem{example}[theorem]{Example}

\newtheorem{remark}[theorem]{Remark}

\newtheorem{lemma}[theorem]{Lemma}
\newtheorem{problem}[theorem]{Problem}
\newtheorem{definition}[theorem]{Definition}
\newtheorem{conjecture}{Conjecture}
\numberwithin{equation}{section}

\DeclareMathOperator{\sign}{sign\,\!}
\DeclareMathOperator{\comp}{\mathcal{C}omp\,\!}
\DeclareMathOperator{\card}{card\,\!}
\begin{document}
\title[Sharp Bohnenblust--Hille constants]{Towards sharp Bohnenblust--Hille
constants}
\author[D. Pellegrino and E. Teixeira]{Daniel Pellegrino and Eduardo V.
Teixeira}
\address{Departamento de Matem\'{a}tica \\
\indent Universidade Federal da Para\'{\i}ba \\
\indent 58.051-900 - Jo\~{a}o Pessoa, Brazil.}
\email{pellegrino@pq.cnpq.br and dmpellegrino@gmail.com}
\address{Departamento de Matem\'{a}tica \\
\indent Universidade Federal do Cear\'{a} \\
\indent Fortaleza, Cear\'{a}, Brazil.}
\email{teixeira@mat.ufc.br}
\thanks{The authors are supported by CNPq.}
\maketitle

\begin{abstract}
We investigate the optimality problem associated with the best constants in
a class of Bohnenblust--Hille type inequalities for $m$--linear forms. While germinal estimates indicated an exponential growth, in this work we provide strong evidences to the conjecture that the sharp constants in the classical Bohnenblust--Hille inequality are universally
bounded, irrespectively of the value of $m$; hereafter referred as the
\textit{Universality Conjecture}. In our approach, we introduce the {notions
of entropy and complexity}, designed to measure, to some extent, the
complexity of such optimization problems.  We show that the notion of entropy is critically
connected to the Universality Conjecture; for instance, that if the entropy
grows at most exponentially with respect to $m$, then the optimal constants of the $m$%
--linear Bohnenblust--Hille inequality for real scalars are indeed bounded universally in $m$.
It is likely that indeed the entropy grows as $4^{m-1}$, and in this
scenario, we show that the optimal constants are precisely $2^{1-\frac{1}{m}%
} $. In the bilinear case, $m=2$, we
show that any extremum of the Littlewood's $4/3$-inequality has entropy $4$
and complexity $2$, and thus we are able to classify all extrema of the problem. We also prove that, for any mixed $\left( \ell _{1},\ell _{2}\right) $%
--Littlewood inequality, the entropy do grow exponentially and the sharp
constants for such a class of inequalities are precisely $(\sqrt{2})^{m-1}$.
In addition to the {notions of entropy and complexity}, the approach we
develop in this work makes decisive use of a family of strongly
non-symmetric $m$--linear forms, which has further consequences to the
theory, as we explain herein.
\end{abstract}

\tableofcontents

\section{Introduction}

Let $\mathbb{K}$ denote the real or the complex scalar field. Given a
positive integer $m$, the Bohnenblust--Hille inequality \cite{bh} assures
the existence of a constant $B_{\mathbb{K},m}\geq 1$ such that
\begin{equation}
\left( \sum_{j_{1},\cdots,j_{m}=1}^{\infty }\left\vert
T(e_{j_{1}},\cdots,e_{j_{m}})\right\vert ^{\frac{2m}{m+1}}\right) ^{\frac{m+1%
}{2m}}\leq B_{\mathbb{K},m}\left\Vert T\right\Vert ,  \label{u88}
\end{equation}%
for all continuous $m$--linear forms $T\colon c_{0}\times \cdots \times
c_{0}\rightarrow \mathbb{K}$. Restricting \eqref{u88} to the case $m=2$ one
recovers the famous Littlewood's $4/3$ inequality \cite{LLL}.
Bohnenblust--Hille inequality is an elegant, far-reaching pearl of classical
analysis; for connections with other fields of research, we refer to \cite%
{boas, defseip, montanaro, seip} and references therein.

The investigation of the sharp constants in the Bohnenblust--Hille
inequality, namely the optimization problem
\begin{equation*}
B_{\mathbb{K},m}:=\inf \left\{ \left( \sum_{j_{1},\cdots ,j_{m}=1}^{\infty
}\left\vert T(e_{j_{1}},\cdots ,e_{j_{m}})\right\vert ^{\frac{2m}{m+1}%
}\right) ^{\frac{m+1}{2m}}\!\!\!\!\!,\text{ among all \textit{m}--linear
forms }T,\text{ with }\Vert T\Vert =1\right\}
\end{equation*}%
is fundamental in many aspects of the theory, and a rather challenging
mathematical problem. Following classical terminology, a minimum $T$ for the
minimization problem above is called an extremum, or an $m$--linear extremum
form. It has been known since the seminal work of Bohnenblust--Hille, \cite%
{bh}, that
\begin{equation*}
B_{\mathbb{K},m}\leq m^{\frac{m+1}{2m}}\left( \sqrt{2}\right) ^{m-1},
\end{equation*}%
for any $m\geq 1$. It was just quite recently that upper bounds for the
Bohnenblust--Hille inequality were refined, see for instance \cite{bayart}
and references therein. By means of interpolations, that is, clever usage of
H\"{o}lder inequality for mixed sums, and with the knowledge of optimal
constants in the Khinchin inequality --- \cite{haa, szarek} for real scalars
and \cite{ko} for complex scalars --- it was proved the existence of
constants $\kappa _{1},\kappa _{2}>0$ such that
\begin{eqnarray}
B_{\mathbb{R},m} &<&\kappa _{1}\cdot m^{\frac{2-\log 2-\gamma }{2}}, \\
B_{\mathbb{C},m} &<&\kappa _{2}\cdot m^{\frac{1-\gamma }{2}},
\end{eqnarray}%
where $\gamma $ is the Euler-Mascheroni constant, and since
\begin{equation*}
m^{\frac{2-\log 2-\gamma }{2}}\approx m^{0.36482};\quad m^{\frac{1-\gamma }{2%
}}\approx m^{0.21139},
\end{equation*}%
the sharp constants grow (at most) sub-linearly with respect to $m$. The
problem of finding good lower bounds for $B_{\mathbb{K},m}$ turns out to be
also delicate. Despite many analytic and numeric attempts, the up-to-now
best known lower bounds for $B_{\mathbb{R},m}$ are still $2^{1-\frac{1}{m}}$%
. In the complex case, nothing is known besides the trivial estimate $B_{%
\mathbb{C},m}\geq 1$.

Recently, the Bohnenblust--Hille inequality has been proved to be part of a
much more general class of inequalities, see for instance \cite{a}. More
than mere generalizations, these broader classes of inequalities reveal
importance nuances hidden in the original one. In particular, significant
advances in the theory can be acquired by the following extended version of
the inequality, see \cite{a}:

\begin{theorem}
\label{ujm}Let $m\geq 2$ be a positive integer and $(q_{1},\cdots,q_{m})\in %
\left[ 1,2\right] ^{m}.$ The following assertions are equivalent:

\begin{itemize}
\item[(A)] $(q_{1},\cdots,q_{m})$ satisfies
\begin{equation}
\frac{1}{q_{1}}+\cdots +\frac{1}{q_{m}}\leq \frac{m+1}{2}.  \label{ppoo}
\end{equation}

\item[(B)] \textit{There exists a constant} $C_{\left(
q_{1},\cdots,q_{m}\right) m}^{\mathbb{K}}\geq 1$ \textit{such that}
\begin{equation}
\left( \sum_{j_{1}=1}^{\infty }\left( \sum_{j_{2}=1}^{\infty }\left( \cdots
\left( \sum_{j_{m}=1}^{\infty }\left\vert
T(e_{j_{1}},\cdots,e_{j_{m}})\right\vert ^{q_{m}}\right) ^{\frac{q_{m-1}}{%
q_{m}}}\cdots \right) ^{\frac{q_{2}}{q_{3}}}\right) ^{\frac{q_{1}}{q_{2}}%
}\right) ^{\frac{1}{q_{1}}}\leq C_{\left( q_{1},\cdots,q_{m}\right) m}^{%
\mathbb{K}}\left\Vert T\right\Vert,  \label{g33}
\end{equation}%
\textit{for all continuous }$m$\textit{--linear forms }$T\colon c_{0}\times
\cdots \times c_{0}\rightarrow \mathbb{K}$.
\end{itemize}
\end{theorem}

\bigskip

When condition (A) is verified, $(q_{1},\cdots,q_{m})\in \left[ 1,2\right]
^{m}$ is said to be a Bohnenblust--Hille exponent. Hereafter, we will
essentially deal with the case $\mathbb{K}=\mathbb{R}$. It is particularly
interesting for our our purposes the mixed $\left( \ell _{1},\ell
_{2}\right) $--Littlewood inequalities, which refers, in inequality (\ref%
{g33}), to the exponents
\begin{equation*}
\left\{ \left( 1,2,\cdots ,2\right) ,\left( 2,1,\cdots ,2\right) ,\cdots
\left( 2,2,\cdots ,1\right) \right\} .
\end{equation*}%
Henceforth the mixed $\left( \ell _{1},\ell _{2}\right) $-Littlewood
inequalities comprise the existence of positive constants
\begin{equation*}
C_{\left( 1,2,\cdots ,2\right) m},\,C_{\left( 2,1,\cdots ,2\right)
m},\,\cdots ,\,C_{\left( 2,2,\cdots ,2,1\right) m}
\end{equation*}%
such that
\begin{equation}
\left\{
\begin{array}{l}
\displaystyle\sum\limits_{j_{i}=1}^{\infty }\left( \sum_{j_{1},\cdots
,j_{i-1},j_{i+1},\cdots ,j_{m}=1}^{\infty }\left\vert U(e_{j_{1}},\cdots
,e_{j_{m}})\right\vert ^{2}\right) ^{\frac{1}{2}}\leq C_{\left( 1,2,\cdots
,2\right) m}\left\Vert U\right\Vert , \\
\vdots \\
\displaystyle\left( \sum\limits_{j_{1},\cdots j_{i-1},j_{i+1},\cdots
j_{m}=1}^{\infty }\left( \sum_{j_{i}=1}^{\infty }\left\vert
U(e_{j_{1}},\cdots ,e_{j_{m}})\right\vert \right) ^{2}\right) ^{1/2}\leq
C_{\left( 2,2,\cdots ,2,1\right) m}\left\Vert U\right\Vert ,%
\end{array}%
\right.  \label{eh}
\end{equation}%
for all continuous $m$--linear forms $U\colon c_{0}\times \cdots \times
c_{0}\rightarrow \mathbb{R}$, and all $i=1,\cdots, m$. For the $\left( \ell
_{1},\ell _{2}\right) $--Littlewood inequalities, it has been proved that
\begin{equation*}
C_{\left( 1,2,\cdots ,2\right) m}=C_{\left( 2,1,2,\cdots ,2\right) m}=\left(
\sqrt{2}\right) ^{m-1}
\end{equation*}%
gives the sharp constant. While the proof of this fact, see \cite{daniel},
relies on the powerful analysis of Haagerup \cite{haa} for finding the best
constants of the Khinchin inequality, the strategy cannot be carried out for
the other exponents, see \cite{dia}. For instance, in the case of the
multiple exponent $\left( 2,\cdots ,2,1\right) $, previous methods for
finding optimal constants simply yield that $\sqrt{2}$ is a lower bound.

In this paper, through a novel approach, we finally show that the optimal
constants of all the mixed $\left( \ell _{1},\ell _{2}\right) $--Littlewood
inequalities are indeed $\left ( \sqrt{2}\right)^{m-1}$, that is,
\begin{equation}
C_{\left( 1,2,\cdots ,2\right) m} = C_{\left( 2,1,2, \cdots ,2\right)m} =
\cdots = C_{\left( 2,2,\cdots 2,1\right) m}=\left( \sqrt{2}\right)^{m-1}.
\label{ki}
\end{equation}

Despite of the infinite-dimensional nature of the Bohnenblust--Hille
inequalities, that is, the $m$--linear forms act on $c_{0}$, it is quite
revealing to note that all $m$--linear extrema of the inequalities above
mentioned are composed by precisely $4^{m-1}$ monomials. Such observation
raises up the following problem, which seems to play a central role in the
theory: \textit{\ how many monomials are needed to create an }$m$\textit{%
--linear extremum form, i.e., an $m$--linear form which makes the optimal
constant of the Bohnenblust--Hille inequality to be attained?}

To handle this problem we introduce the notion of entropy (formally defined
in the next section) as the minimal number of monomials needed to assemble
an extremum of the respective inequality. For instance, the entropy of the
exponents of the class of mixed $\left( \ell _{1},\ell _{2}\right) $%
--Littlewood inequalities will be less than or equal to $4^{m-1}$. We expect
that the entropy of the classical Bohnenblust--Hille inequality behaves
similarly. A stunning implication of such statement is that the optimal
constants of the Bohnenblust--Hille inequality would then be uniformly
bounded with respect to $m$. Indeed, we will prove in this work that,
restricted to $m$--linear forms composed by the combination of up to $%
4^{m-1} $ monomials, the optimal constants of the Bohnenblust--Hille
inequality for real scalars are precisely $2^{1-\frac{1}{m}}.$ This is
direct evidence to support the following striking conjecture:

\medskip

\noindent \textbf{Universality Conjecture.} \textit{The optimal constants in
the Bohnenblust--Hille inequality are universally bounded, irrespectively of
the value of $m$. In the real case, the best constants should be precisely $%
2^{1-\frac{1}{m}}.$ }

\medskip

Our approach to establish \eqref{ki} makes decisive use of highly
non-symmetric $m$--linear forms. As a possible predicament, the most
efficient tools known up-to-now to produce upper estimates for the
Bohnenblust--Hille constants, namely interpolations, are probably not suited
to reach the sharp estimates. We elaborate such considerations in the last
section of this article.

\section{The notions of entropy {and complexity}}

The optimality problem for the classical Littlewood's $4/3$ inequality, i.e.
the case $m=2$ in the Bohnenblust--Hille inequality, is relatively well
understood in the literature; the optimal constant satisfying (\ref{u88})
for the case of real scalars is $\sqrt{2}$. Furthermore, it is attained by
the bilinear form
\begin{equation}
T_{2}(x,y)=x_{1}y_{1}+x_{1}y_{2}+x_{2}y_{1}-x_{2}y_{2}  \label{att L}
\end{equation}%
--- a simple-looking $2$--linear form comprised of only $4$ monomials.
Besides, essentially no other significantly different extremum is known.
%(see \cite{diniz} for the strategic $m$--linear forms used to provide lower bounds for the $m$--linear Bohnenblust--Hille inequality).
There is an obvious, empirical relation between the difficulty of
establishing the sharp constants in the Bohnenblust--Hille and the algebraic
complexity of prospective extrema. Hence, the following definition is
suitable for the purposes of this study and, we deem, it worths further
investigation.

\begin{definition}
Let $m$ be a positive integer and $\left( q_{1},\cdots,q_{m}\right) $ be a
Bohnenblust--Hille exponent. The entropy of $\left(
q_{1},\cdots,q_{m}\right) $ is defined as%
\begin{equation*}
\ent_{m}^{\mathbb{K}}\left( q_{1},\cdots,q_{m}\right) =\inf
\{card(i_{1},\cdots,i_{m}):\alpha _{i_{1},\cdots,i_{m}}\neq 0\},
\end{equation*}%
where this infimum is taken over all continuous $m$--linear forms $T \colon
c_{0}\times \cdots \times c_{0}\rightarrow \mathbb{K}$ defined by%
\begin{equation*}
T\left( x^{(1)},\cdots,x^{(m)}\right)
=\sum\limits_{i_{1},\cdots,i_{m}=1}^{\infty }\alpha
_{i_{1},\cdots,i_{m}}x_{i_{1}}^{(1)}\cdots x_{i_{m}}^{(m)}
\end{equation*}%
such that the optimal constants of the Bohnenblust--Hille inequality with
multiple exponents $\left( q_{1},\cdots,q_{m}\right) $ is attained by $T$.
\end{definition}

The entropy, $\ent_{m}^{\mathbb{K}}\left( q_{1},\cdots,q_{m}\right) $,
measures, henceforth, the minimal number of monomials necessary to assemble
an extremal $m$--linear form, for which the optimal constant satisfying (\ref%
{g33}) is reached. Since we will focus our attention in the case of real
scalars, hereafter we denote $\ent_{m}^{\mathbb{R}}\left(
q_{1},\cdots,q_{m}\right) $ just by $\ent_{m}\left(
q_{1},\cdots,q_{m}\right).$ We remark that for $m>2$, it is not known in
general the existence of extrema. In this untoward case, we define $\ent%
_{m}^{\mathbb{K}}\left( q_{1},\cdots,q_{m}\right) =\infty.$

Our earlier discussion regarding the explicit $2$-form extremum for the
Littlewood's $4/3$ inequality, \eqref{att L}, reads in the formalism of
entropy as
\begin{equation*}
\ent_{2}\left( \frac{4}{3},\frac{4}{3}\right) \leq 4,
\end{equation*}%
and it is in fact simple to verify that equality holds, i.e., $\ent%
_{2}\left( \frac{4}{3},\frac{4}{3}\right) =4$. By similar reasoning, one
deduces that%
\begin{equation*}
\ent_{2}\left( 1,2\right) =\ent_{2}\left( 2,1\right) =4.
\end{equation*}

By obvious reason, $\ent_{1}\left( 1\right) =1.$ We will prove that if the
entropies grow like $4^{m-1}$ (or at least exponentially) with respect to $m$%
, then the optimal $m$--linear Bohnenblust--Hille constants, for the case of
real scalars, should then be bounded (precisely $2^{1-\frac{1}{m}}$ if the
entropies of the Bohnenblust--Hille exponents are $4^{m-1})$, confirming
henceforth the Conjecture regarding sharp estimates for such class of
inequalities.

We conclude this section with a comment on the class of $m$--linear forms
composed by the sum of monomials whose coefficients are $\pm 1$ (unimodular
monomials), as it plays a relevant -- probably central -- role in the
theory. By way of example, we recall that the proof of the optimality of
Theorem \ref{ujm} or, more precisely, the proof of (B)$\Rightarrow $(A) uses
only $m$--linear forms composed by the sum of unimodular monomials via the
Kahane--Salem--Zygmund inequality; in other words we have (B)$\Rightarrow $%
(A) even if restrict the inequalities just to $m$--linear forms composed by
unimodular monomials. As further evidence of the importance of $m$--linear
forms composed by unimodular monomials, it will be shown later in this work
that the optimal constants of any mixed $\left( \ell _{1},\ell _{2}\right) $%
-Littlewood inequalities are attained by $m$--linear forms with coefficients
$\pm 1$ (see also \ref{att L}).

\begin{example}
\label{7b} In the case of bilinear forms $T:c_{0}\times c_{0}\rightarrow
\mathbb{R}$ of the type $%
T(x,y)=ax_{1}y_{1}+bx_{1}y_{2}+cx_{2}y_{1}+dx_{2}y_{1},$ using the geometry
of the closed unit ball of the space of bilinear forms on $\ell_{\infty
}^{2}\times \ell_{\infty }^{2},$ it was proved in \cite{ccc} that the extrema
are essentially the bilinear forms of the type (\ref{att L}). If we restrict
our attention to $m$--linear forms with unimodular monomials, the conjecture
that $2^{1-\frac{1}{m}}$ is sharp constant for the Bohnenblust--Hille
inequality is equivalent (using the Krein-Milman Theorem) to the following
question in number theory: For all positive integers $m$ and $n,$ $%
i_{1},\cdots ,i_{m}\in \{1,\cdots ,n\}$ and $\delta _{i_{1}\cdots i_{m}}\in
\left\{ 0,1,-1\right\} ,$ does the following estimate
\begin{equation}
\frac{\left( \card\left\{ \left( i_{1}\cdots i_{m}\right) :\delta
_{i_{1}\cdots i_{m}}\neq 0\right\} \right) ^{\frac{m+1}{2m}}}{2^{1-\frac{1}{m%
}}}\leq \max \left\{ \sum\limits_{i_{1},\cdots ,i_{m}=1}^{n}\delta
_{i_{1}\cdots i_{m}}x_{i_{1}}^{(1)}\cdots x_{i_{m}}^{(m)}:\text{ }%
x_{i_{k}}^{(j)}\in \left\{ -1,1\right\} \right\}   \label{NT}
\end{equation}%
hold true? Treating the conjecture in this format brings computational
advantages. For instance, equality in \eqref{NT} holds when $\delta
_{i_{1}\cdots i_{m}}$ follows the pattern of the strategic \textquotedblleft
classical" $m$--linear forms. For $m=2$, we know that \eqref{NT} is true.
\end{example}

We conclude this present section by introducing the notion of complexity,
which will also play an important role in the analyses to be carried out in
the next sections.

\medskip

Let $T\colon c_{0}\times \cdots \times c_{0}\rightarrow \mathbb{K}$ be given
by%
\begin{equation*}
T\left( x^{(1)},\cdots ,x^{(m)}\right) =\sum\limits_{i_{1},\cdots
,i_{m}=1}^{\infty }\alpha _{i_{1},\cdots ,i_{m}}x_{i_{1}}^{(1)}\cdots
x_{i_{m}}^{(m)}
\end{equation*}%
and define, for all fixed $k${$=1,\cdots ,m,$ and fixed }$i_{k}\in \mathbb{N}
$, {the set%
\begin{equation}
S_{k}^{(i_{k})}(T)=\left\{ \left( i_{1}\cdots i_{m}\right) :\alpha
_{i_{1},\cdots ,i_{m}}\neq 0\right\} .  \label{9009}
\end{equation}%
}We define the \textit{complexity} of $T$, and denote by $\comp(T)$, as%
\begin{equation}
\comp\left( T\right) =\sup_{k=1,\cdots ,m}~\sup_{i_{k}\in \mathbb{N}}\card%
\left( S_{k}^{(i_{k})}(T)\right) ,  \label{9000}
\end{equation}

For instance, if for $j=1,2,$
\begin{equation*}
T_{j}\colon c_{0}\times c_{0}\rightarrow \mathbb{K}
\end{equation*}%
are given by
\begin{equation*}
T_{1}(x,y)=x_{1}\sum\limits_{j=1}^{\infty }\frac{y_{j}}{2^{j}},
\end{equation*}%
and%
\begin{equation*}
T_{2}(x,y)=x_{1}y_{1}+x_{1}y_{2}+x_{2}y_{3},
\end{equation*}%
then
\begin{eqnarray*}
S_{1}^{(1)}(T_{1}) &=&\left\{ \left( 1,j\right) :j\in \mathbb{N}\right\} , \\
S_{1}^{(k)}(T_{1}) &=&\varnothing \text{ for }k>1, \\
S_{2}^{(k)}(T_{1}) &=&\left\{ \left( 1,k\right) \right\} \text{ for all }k,
\end{eqnarray*}%
and%
\begin{eqnarray*}
S_{1}^{(1)}(T_{2}) &=&\left\{ \left( 1,1\right) ,\left( 1,2\right) \right\} ,
\\
S_{1}^{(2)}(T_{2}) &=&\left\{ \left( 2,3\right) \right\} , \\
S_{2}^{(1)}(T_{2}) &=&\left\{ \left( 1,1\right) \right\} , \\
S_{2}^{(2)}(T_{2}) &=&\left\{ \left( 1,2\right) \right\} , \\
S_{2}^{(3)}(T_{2}) &=&\left\{ \left( 2,3\right) \right\} ,
\end{eqnarray*}%
and hence
\begin{equation*}
\comp(T_{1})=\infty \text{ and }\comp(T_{2})=2.
\end{equation*}%
In essence, the complexity of a $m$-linear form measures the biggest
possible degree of combination between the variables.

In the next section we classify all the extrema of the Littlewood's $4/3$%
-inequality, that is the $2$-linear Bohnenblust--Hille inequality.

%%%%%%%%%%%%%%%%%%%%%%%%%%%%%%%%%%%%%%%%%%%

\section{Classification of all extrema for Littlewood's $4/3$ inequality
\label{8zzz}}

As it has been previously mentioned, all known extrema for Littlewood's $4/3$
inequality up to date are bilinear forms like (\ref{att L}). In this section
we show that this is in fact essentially the unique extremum for
Littlewood's $4/3$ inequality. Besides its own interest, this result hints
out the possible pattern for the extrema in the $m$-linear case (see Section %
\ref{g} for details). As we will show, this would ultimately prove that, in
fact, the optimal constants of the $m$-linear Bohnenblust--Hille inequality
are $2^{1-\frac{1}{m}}$, and thus bounded with respect to $m$.

Before we continue, for the sake of the reader's convenience, let us recall
the Khinchin inequality. If $r_n(t)$ denote the Rademacher functions,
\begin{equation*}
r_{n}(t):=\sign\left( \sin 2^{n}\pi t\right),
\end{equation*}%
the Khinchin inequality asserts that for any $p>0$ there are constants $%
A_{p},B_{p}>0$ such that
\begin{equation}
A_{p}\left( \sum\limits_{j=1}^{\infty }\left\vert a_{j}\right\vert
^{2}\right) ^{\frac{1}{2}}\leq \left( \int\limits_{0}^{1}\left\vert
\sum\limits_{j=1}^{\infty }r_{j}(t)a_{j}\right\vert ^{p}dt\right) ^{\frac{1}{%
p}}\leq B_{p}\left( \sum\limits_{j=1}^{\infty }\left\vert a_{j}\right\vert
^{2}\right) ^{\frac{1}{2}}  \label{65}
\end{equation}%
for all sequence of scalars $\left( a_{i}\right) _{i=1}^{\infty }.$ The best
constants $A_{p}$ are known (see \cite{haa, szarek}):

\begin{itemize}
\item $\displaystyle A_{p}=\sqrt{2}\left( \frac{\Gamma \left( \frac{1+p}{2}%
\right) }{\sqrt{\pi }}\right) ^{\frac{1}{p}}$ if $p\geq p_{0}\cong 1.8474$;

\item $\displaystyle A_{p}=2^{\frac{1}{2}-\frac{1}{q}}$ if $p<p_0$.
\end{itemize}

\medskip

The critical exponent $p_{0}$ is the unique real number in $p_{0}\in (1,2)$
verifying
\begin{equation*}
\Gamma \left( \frac{p_{0}+1}{2}\right) =\frac{\sqrt{\pi }}{2}.
\end{equation*}

\begin{lemma}
\label{7vg}If $T(x,y)=\sum\limits_{i,j=1}^{\infty }a_{ij}x_{i}y_{j}$ is an
extremum for Littlewood's $4/3$ inequality, then%
\begin{equation}
\comp\left( T\right) =2.  \label{8h}
\end{equation}%
Moreover, for all $\left( i_{1,0}j_{1}\right) ,\left( i_{1,0}j_{2}\right)
\in S_{1}^{(i_{1,0})}$ and $\left( i_{1}i_{2,0}\right) ,\left(
i_{2}i_{2,0}\right) \in S_{2}^{(i_{2,0})}$, we have%
\begin{equation}
\left\vert a_{i_{1,0}j_{1}}\right\vert =\left\vert
a_{i_{1,0}j_{2}}\right\vert  \label{8h2}
\end{equation}%
and%
\begin{equation}
\left\vert a_{i_{1}i_{2,0}}\right\vert =\left\vert
a_{i_{2}i_{2,0}}\right\vert .  \label{8h3}
\end{equation}
\end{lemma}

\begin{proof}
By the Khinchin Inequality we know that
\begin{equation*}
\left( \sum\limits_{j=1}^{\infty }\left\vert a_{j}\right\vert ^{2}\right) ^{%
\frac{1}{2}}\leq \sqrt{2}\left( \int\limits_{0}^{1}\left\vert
\sum\limits_{j=1}^{\infty }r_{j}(t)a_{j}\right\vert dt\right) ,
\end{equation*}%
and the equality holds if and only if $\left( a_{j}\right) _{j=1}^{\infty
}=\alpha \left( \pm e_{i}\pm e_{j}\right) $ for some $\alpha \neq 0$ and $%
i\neq j$ (see (\cite{szarek}). In particular, if (\ref{8h}) or (\ref{8h2})
or (\ref{8h3}) fails, then%
\begin{equation*}
\left( \sum\limits_{j=1}^{\infty }\left\vert a_{j}\right\vert ^{2}\right) ^{%
\frac{1}{2}}<\sqrt{2}\left( \int\limits_{0}^{1}\left\vert
\sum\limits_{j=1}^{\infty }r_{j}(t)a_{j}\right\vert dt\right) .
\end{equation*}%
In fact, if $T:c_{0}\times c_{0}\rightarrow \mathbb{R}$ is given by
\begin{equation*}
T(x,y)=\sum\limits_{i,j=1}^{\infty }a_{ij}x_{i}y_{j}
\end{equation*}%
with%
\begin{equation*}
0<\card\left( S_{2}^{(j_{0})}\right) \neq 2
\end{equation*}%
for some $j_{0},$ we have%
\begin{eqnarray*}
&&\sum\limits_{j=1}^{\infty }\left( \sum\limits_{i=1}^{\infty }\left\vert
T(e_{i},e_{j}\right\vert ^{2}\right) ^{\frac{1}{2}} \\
&=&\left( \sum\limits_{i=1}^{\infty }\left\vert
T(e_{i},e_{j_{0}})\right\vert ^{2}\right) ^{\frac{1}{2}}+\sum\limits_{j\neq
j_{0}}\left( \sum\limits_{i=1}^{\infty }\left\vert T(e_{i},e_{j})\right\vert
^{2}\right) ^{\frac{1}{2}} \\
&<&\sqrt{2}\left( \int\limits_{0}^{1}\left\vert \sum\limits_{i=1}^{\infty
}r_{i}(t)T(e_{i},e_{j_{0}})\right\vert dt\right) +\sum\limits_{j\neq j_{0}}%
\sqrt{2}\left( \int\limits_{0}^{1}\left\vert \sum\limits_{i=1}^{\infty
}r_{i}(t)T(e_{i},e_{j})\right\vert dt\right)  \\
&=&\sqrt{2}\left( \int\limits_{0}^{1}\sum\limits_{j=1}^{\infty }\left\vert
T\left( \sum\limits_{i=1}^{\infty }r_{i}(t)e_{i},e_{j}\right) \right\vert
dt\right)
\end{eqnarray*}%
and thus%
\begin{equation}
\sum\limits_{j=1}^{\infty }\left( \sum\limits_{i=1}^{\infty }\left\vert
T(e_{i},e_{j})\right\vert ^{2}\right) ^{\frac{1}{2}}<\sqrt{2}\left\Vert
T\right\Vert .  \label{set44}
\end{equation}%
The other cases are similar. On the other hand, it is well known that%
\begin{equation}
\left( \sum\limits_{j=1}^{\infty }\left( \sum\limits_{i=1}^{\infty
}\left\vert T(e_{i},e_{j})\right\vert ^{1}\right) ^{\frac{1}{1}\cdot
2}\right) ^{\frac{1}{2}}\leq \sqrt{2}\left\Vert T\right\Vert .  \label{set55}
\end{equation}%
Applying (\ref{set44}), (\ref{set55}) and the H\"{o}lder inequality for
mixed sums we conclude that%
\begin{equation*}
\left( \sum_{i,j=1}^{\infty }\left\vert T(e_{i},e_{j})\right\vert ^{\frac{4}{%
3}}\right) ^{3/4}<\sqrt{2}\left\Vert T\right\Vert .
\end{equation*}%
Therefore, if $T$ is an extremum we have (\ref{8h}), (\ref{8h2}) and (\ref%
{8h3}).
\end{proof}

There are three types of bilinear forms on $c_{0}\times c_{0}$ that verify (%
\ref{8h}), (\ref{8h2}) and (\ref{8h3}):

\bigskip

\noindent \textbf{Type I}: there are $\left( N_{j}\right) _{j=1}^{s}$ and $%
\left( M_{j}\right) _{j=1}^{s},$ each of one of them a family of pairwise
disjoint positive integers with $\card\left( N_{k}\right) =\card\left(
M_{k}\right) =2$ for all $k$ and $s\in \lbrack 1,\infty ]$, such that
\begin{equation*}
T(x,y)=\sum\limits_{k=1}^{s}\left( \overset{\text{elementary term of type I}}%
{\overbrace{\sum\limits_{\left( i,j\right) \in N_{k}\times
M_{k}}a_{ij}^{(k)}x_{i}y_{j}}}\right) ,
\end{equation*}%
with $a_{ij}^{(k)}\neq 0$ and, for all $k,$ we have $\left\vert
a_{ij}^{(k)}\right\vert =\left\vert a_{ij}^{(k)}\right\vert $ whenever $%
\left( i,j\right) \in N_{k}\times M_{k}.$ Note that the elementary terms are
composed by $4$ monomials with the same coefficient. An illustration of
bilinear form of type I is:%
\begin{equation*}
T_{1}(x,y)=\left( \overset{\text{elementary term of type I}}{\overbrace{\pm
ax_{1}y_{1}\pm ax_{1}y_{2}\pm ax_{2}y_{1}\pm ax_{2}y_{1}}}\right) +\left(
\overset{\text{elementary term of type I}}{\overbrace{\pm bx_{3}y_{3}\pm
bx_{3}y_{4}\pm bx_{4}y_{3}\pm bx_{4}y_{4}}}\right) ,
\end{equation*}%
with $a\neq 0$ and $b\neq 0$.

\bigskip

\noindent \textbf{Type II}: In this case we consider sums of bigger
monomials (recall that the complexity is always $2$). The simplest case of
bilinear forms of type II is a bilinear form like
\begin{equation}
T_{2}(x,y)=\overset{\text{elementary term of type II}}{\overbrace{\left( \pm
ax_{1}y_{1}\pm ax_{1}y_{2}\right) +\left( \pm ax_{3}y_{1}\pm
ax_{4}y_{2}\right) +\left( \pm ax_{3}y_{5}\pm ax_{4}y_{5}\right) }},
\label{799}
\end{equation}%
with $a\neq 0.$ In this case it is simple to prove that%
\begin{equation*}
\left\Vert T_{2}\right\Vert \geq 4\left\vert a\right\vert
\end{equation*}%
and since%
\begin{equation*}
\frac{\left( 6\left\vert a\right\vert ^{4/3}\right) ^{3/4}}{\left\Vert
T_{2}\right\Vert }\leq \frac{6^{3/4}}{4}<1<\sqrt{2}
\end{equation*}%
we do not have an extremum. Another illustration is%
\begin{equation}
T_{2^{\prime }}(x,y)=\overset{\text{elementary term of type II}}{\overbrace{%
\left( \pm ax_{1}y_{1}\pm ax_{1}y_{2}\right) +\left( \pm ax_{3}y_{1}\pm
ax_{4}y_{2}\right) +\left( \pm ax_{3}y_{5}\pm ax_{4}y_{6}\right) +\left( \pm
ax_{5}y_{5}\pm ax_{5}y_{6}\right) }},  \label{801}
\end{equation}%
with $a\neq 0.$ In this case%
\begin{equation*}
\left\Vert T_{2^{\prime }}\right\Vert \geq 6\left\vert a\right\vert
\end{equation*}%
and since%
\begin{equation*}
\frac{\left( 8\left\vert a\right\vert ^{4/3}\right) ^{3/4}}{\left\Vert
T_{2^{\prime }}\right\Vert }\leq \frac{8^{3/4}}{6}<0.8<\sqrt{2}
\end{equation*}%
we, again, do not have an extremum and so on. We may also have sums of
elementary\ terms. For instance:%
\begin{eqnarray*}
T_{2^{\prime \prime }}(x,y) &=&\overset{\text{elementary term of type II}}{%
\left[ \overbrace{\left( \pm ax_{1}y_{1}\pm ax_{1}y_{2}\right) +\left( \pm
ax_{3}y_{1}\pm ax_{4}y_{2}\right) +\left( \pm ax_{3}y_{5}\pm
ax_{4}y_{5}\right) }\right] } \\
&&+\overset{\text{elementary term of type II}}{\left[ \overbrace{\left( \pm
bx_{6}y_{6}\pm bx_{6}y_{7}\right) +\left( \pm bx_{7}y_{6}\pm
bx_{8}y_{7}\right) +\left( \pm bx_{7}y_{8}\pm bx_{8}y_{9}\right) +\left( \pm
bx_{9}y_{8}\pm bx_{9}y_{9}\right) }\right] },
\end{eqnarray*}%
with $a\neq 0$ and $b\neq 0.$ This case is even farther from being an
extremum, since%
\begin{equation*}
\left\Vert T_{2^{\prime \prime }}\right\Vert \geq 4\left\vert a\right\vert
+6\left\vert b\right\vert
\end{equation*}%
and thus%
\begin{eqnarray*}
\frac{\left( 6\left\vert a\right\vert ^{4/3}+8\left\vert b\right\vert
^{4/3}\right) ^{3/4}}{\left\Vert T_{2^{\prime \prime }}\right\Vert } &\leq &%
\frac{\left( 6\left\vert a\right\vert ^{4/3}+8\left\vert b\right\vert
^{4/3}\right) ^{3/4}}{4\left\vert a\right\vert +6\left\vert b\right\vert } \\
&\leq &\frac{6^{3/4}\left\vert a\right\vert +8^{3/4}\left\vert b\right\vert
}{4\left\vert a\right\vert +6\left\vert b\right\vert } \\
&\leq &\max \left\{ \frac{6^{3/4}}{4},\frac{8^{3/4}}{6}\right\} <\sqrt{2}.
\end{eqnarray*}%
A simple, though tedious, argument shows that any such combination of
elementary terms does not provide extrema.\bigskip

\noindent \textbf{Type III}: combinations of elements of the first and
second types. Let $T_{3}=R_{1}+R_{2}$ with $R_{J}\neq 0$ being of type $%
J=1,2.$ Note that since
\begin{equation*}
\comp(T_{3})=\comp\left( R_{1}\right) =\comp\left( R_{2}\right) =2,
\end{equation*}%
there is no overlapping between $R_{1}$ and $R_{2}$ and we have%
\begin{equation*}
\left\Vert T_{3}\right\Vert =\left\Vert R_{1}\right\Vert +\left\Vert
R_{2}\right\Vert .
\end{equation*}%
Thus, since $R_{2}$ is not an extremum,%
\begin{eqnarray*}
\left( \sum_{i,j=1}^{\infty }\left\vert T_{3}(e_{i},e_{j})\right\vert ^{%
\frac{4}{3}}\right) ^{3/4} &\leq &\left( \sum_{i,j=1}^{\infty }\left\vert
R_{1}(e_{i},e_{j})\right\vert ^{\frac{4}{3}}\right) ^{3/4}+\left(
\sum_{i,j=1}^{\infty }\left\vert R_{2}(e_{i},e_{j})\right\vert ^{\frac{4}{3}%
}\right) ^{3/4} \\
&<&\sqrt{2}\left\Vert R_{1}\right\Vert +\sqrt{2}\left\Vert R_{2}\right\Vert
\\
&=&\sqrt{2}\left\Vert T_{3}\right\Vert ,
\end{eqnarray*}%
and thus $T_{3}$ is not an extremum.

\medskip

From the previous considerations we conclude that extrema of Littlewood's $%
4/3$ inequality must satisfy (\ref{8h}), (\ref{8h2}) and (\ref{8h3}) and be
of type I. The following theorem gives a final and complete characterization:

\begin{theorem}
A bilinear form $T$ is an extremum of Littlewood's $4/3$ inequality if and
only if $T$ is written as
\begin{eqnarray*}
T(x,y) &=&2^{-1/2}\left(
x_{i_{1}}y_{i_{2}}+x_{i_{1}}y_{i_{3}}+x_{i_{2}}y_{i_{4}}-x_{i_{2}}y_{i_{4}}%
\right) , \\
T(x,y) &=&2^{-1/2}\left(
x_{i_{1}}y_{i_{2}}+x_{i_{1}}y_{i_{3}}-x_{i_{2}}y_{i_{4}}+x_{i_{2}}y_{i_{4}}%
\right) , \\
T(x,y) &=&2^{-1/2}\left(
x_{i_{1}}y_{i_{2}}-x_{i_{1}}y_{i_{3}}+x_{i_{2}}y_{i_{4}}+x_{i_{2}}y_{i_{4}}%
\right) , \\
T(x,y) &=&2^{-1/2}\left(
-x_{i_{1}}y_{i_{2}}+x_{i_{1}}y_{i_{3}}+x_{i_{2}}y_{i_{4}}+x_{i_{2}}y_{i_{4}}%
\right)
\end{eqnarray*}%
for $i_{1}\neq i_{2}$ and $i_{3}\neq i_{4}.$
\end{theorem}

\begin{proof}
We just need to consider bilinear forms of the type I. Denoting%
\begin{equation*}
T_{k}(x,y)=\sum\limits_{\left( i,j\right) \in N_{k}\times
M_{k}}a_{ij}^{(k)}x_{i}y_{j},
\end{equation*}%
since $\left( N_{j}\right) _{j=1}^{s}$ and $\left( M_{j}\right) _{j=1}^{s}$
are, each of one of them, a family of pairwise disjoint positive integers,
we have%
\begin{equation*}
\left\Vert T\right\Vert =\sum\limits_{k=1}^{s}\left\Vert T_{k}\right\Vert
\end{equation*}%
and
\begin{eqnarray*}
\left( \sum_{i,j=1}^{\infty }\left\vert T(e_{i},e_{j})\right\vert ^{\frac{4}{%
3}}\right) ^{3/4} &=&\left( \sum_{i,j=1}^{\infty }\left\vert
\sum_{k=1}^{s}T_{k}(e_{i},e_{j})\right\vert ^{\frac{4}{3}}\right) ^{3/4} \\
&\leq &\sum_{k=1}^{s}\left( \sum_{i,j=1}^{\infty }\left\vert
T_{k}(e_{i},e_{j})\right\vert ^{\frac{4}{3}}\right) ^{3/4}.
\end{eqnarray*}%
If $T_{k}$ is not extremum for some $k$, we have%
\begin{eqnarray*}
\left( \sum_{i,j=1}^{\infty }\left\vert T(e_{i},e_{j})\right\vert ^{\frac{4}{%
3}}\right) ^{3/4} &\leq &\sum_{k=1}^{s}\left( \sum_{i,j=1}^{\infty
}\left\vert T_{k}(e_{i},e_{j})\right\vert ^{\frac{4}{3}}\right) ^{3/4} \\
&<&\sqrt{2}\sum\limits_{k=1}^{\infty }\left\Vert T_{k}\right\Vert  \\
&=&\sqrt{2}\left\Vert T\right\Vert
\end{eqnarray*}%
and thus $T$ is not extremum. Thus all $T_{k}$ need to be extrema. Recall
that for a given $k,$ the coefficients of the monomials of $T_{k}$ are all
the same in absolute value. So, it is simple to check that
\begin{equation*}
T_{k}(x,y)=\alpha _{k}\left(
x_{i_{k,1}}y_{i_{k,2}}+x_{i_{k,1}}y_{i_{k,3}}+x_{i_{k,4}}y_{i_{k,2}}-x_{i_{k,4}}y_{i_{k,3}}\right) ,
\end{equation*}%
or
\begin{equation*}
T_{k}(x,y)=\alpha _{k}\left(
x_{i_{k,1}}y_{i_{k,2}}+x_{i_{k,1}}y_{i_{k,3}}-x_{i_{k,4}}y_{i_{k,2}}+x_{i_{k,4}}y_{i_{k,3}}\right) ,
\end{equation*}%
or
\begin{equation*}
T_{k}(x,y)=\alpha _{k}\left(
x_{i_{k,1}}y_{i_{k,2}}-x_{i_{k,1}}y_{i_{k,3}}+x_{i_{k,4}}y_{i_{k,2}}+x_{i_{k,4}}y_{i_{k,3}}\right) ,
\end{equation*}%
or
\begin{equation*}
T_{k}(x,y)=\alpha _{k}\left(
-x_{i_{k,1}}y_{i_{k,2}}+x_{i_{k,1}}y_{i_{k,3}}+x_{i_{k,4}}y_{i_{k,2}}+x_{i_{k,4}}y_{i_{k,3}}\right)
\end{equation*}%
for some $\alpha _{k}\neq 0,$ and there is no overlapping between $%
T_{k_{1}},T_{k_{2}}$ for $k_{1}\neq k_{2}.$ We have%
\begin{equation*}
\left( \sum_{i,j=1}^{\infty }\left\vert T_{1}(e_{i},e_{j})+\cdots
+T_{s}(e_{i},e_{j})\right\vert ^{\frac{4}{3}}\right) ^{3/4}=\left(
4\sum_{k=1}^{s}\left\vert \alpha _{k}\right\vert ^{4/3}\right) ^{3/4}
\end{equation*}%
and%
\begin{equation*}
\left\Vert T\right\Vert =\left\Vert T_{1}+\cdots +T_{s}\right\Vert
=2\sum_{k=1}^{s}\left\vert \alpha _{k}\right\vert .
\end{equation*}%
Since%
\begin{equation*}
\frac{\left( 4\sum_{k=1}^{s}\left\vert \alpha _{k}\right\vert ^{4/3}\right)
^{3/4}}{2\sum_{k=1}^{s}\left\vert \alpha _{k}\right\vert }=\sqrt{2},
\end{equation*}%
we conclude that $s=1.$ In fact, if $s>1,$ since $\alpha _{k}\neq 0$ for all
$k,$ we have%
\begin{equation*}
\left( \sum_{k=1}^{s}\left\vert \alpha _{k}\right\vert ^{4/3}\right)
^{3/4}<\sum_{k=1}^{s}\left\vert \alpha _{k}\right\vert ,
\end{equation*}%
and thus%
\begin{equation*}
\frac{\left( \sum_{i,j=1}^{\infty }\left\vert T_{1}(e_{i},e_{j})+\cdots
+T_{s}(e_{i},e_{j})\right\vert ^{\frac{4}{3}}\right) ^{3/4}}{\left\Vert
T_{1}+\cdots +T_{s}\right\Vert }=\frac{\left( 4\sum_{k=1}^{s}\left\vert
\alpha _{k}\right\vert ^{4/3}\right) ^{3/4}}{2\sum_{k=1}^{s}\left\vert
\alpha _{k}\right\vert }<\sqrt{2}
\end{equation*}%
and $T$ is not extremum.
\end{proof}

\begin{remark}
We have defined entropy as the minimal number of monomials needed to be
combined to generate an extremum. As we have seen, for $m=2$ this number is $%
4$, but this is not only a minimum; this is the unique number of monomials
that can be combined to generate an extremum. This is a quite curious
property that may be inherited when $m>2.$
\end{remark}

In the next section we show that the entropy of any mixed $\left( \ell
_{1},\ell _{2}\right) $-Littlewood inequalities does grow as predicted.
Indeed, we will prove the following lower and upper bounds
\begin{equation}
2^{m-1}\leq \ent_{m}(1,2,2,\cdots ,2),\cdots ,\ent_{m}(2,2,\cdots ,2,1)\leq
4^{m-1},  \label{7777}
\end{equation}%
which is definitive step in the proof of (\ref{ki}).

\section{Entropies with exponential growth and sharp constants \label{8nn}}

In this section we deliver a proof of (\ref{ki}) and (\ref{7777}). Next
theorem is the main step in this endeavor, which actually paves the way to
all the other optimal estimates we will obtain in this work.

\begin{theorem}
\label{8800}Let $m\geq 2$ and $i\geq 1$ be integers, then
\begin{equation}
\left( \sum_{j_{2},\cdots,j_{i-1},j_{i+1},\cdots,j_{m}=1}^{\infty }\left(
\sum_{j_{i}=1}^{\infty }\left\vert U(e_{j_{1}},\cdots,e_{j_{m}})\right\vert
\right) ^{2}\right) ^{1/2}\leq \left( \sqrt{2}\right) ^{m-1} \left\Vert
U\right\Vert,  \label{aaa}
\end{equation}%
holds for all continuous real $m$--linear forms $U\colon c_{0}\times \cdots
\times c_{0}\rightarrow \mathbb{R}$. Furthermore,
\begin{equation*}
\left( \sqrt{2}\right) ^{m-1} = C_{\left( 2,2,\cdots,2,1\right)m}
\end{equation*}
is the sharp constant.
\end{theorem}

A consequence of Theorem \ref{8800} is that all sharp constants in the mixed
$\left( \ell _{1},\ell _{2}\right) $--Littlewood inequality is in fact equal
to $\left( \sqrt{2}\right) ^{m-1}.$

\begin{corollary}
\label{21}For all $m\geq 2$, we have%
\begin{equation}
C_{\left( 1,2,\cdots ,2\right) m}=\cdots =C_{\left( 2,2, \cdots ,2,1\right)
m}=\left( \sqrt{2}\right) ^{m-1}  \label{zzz}
\end{equation}%
and%
\begin{equation*}
2^{m-1}\leq \ent_{m}(1,2, \cdots ,2), \cdots ,\ent_{m}(2,2, \cdots ,2,1)\leq
4^{m-1}.
\end{equation*}
\end{corollary}

The rest of the section is devoted to the proofs of Theorem \ref{8800} and
Corollary \ref{21}, which are based on a radical change in the
\textquotedblleft usual\textquotedblright\ strategic $m$--linear forms.

\bigskip We start off by delivering a proof of the following simple
estimate:
\begin{equation}
\left( \sum\limits_{i_{1},\ldots ,i_{m}=1}^{\infty }\left\vert
U(e_{i_{^{1}}},\ldots ,e_{i_{m}})\right\vert ^{2}\right) ^{\frac{1}{2}}\leq
\left\Vert U\right\Vert ,  \label{9m}
\end{equation}%
holds for all continuous $m$--linear forms $U\colon c_{0}\times \cdots
\times c_{0}\rightarrow \mathbb{K}$. This is a consequence of Khinchin
inequality, or cotype if one prefers. Indeed, applying the Khinchin
inequality together with an induction argument, we reach%
\begin{equation*}
\left( \sum\limits_{i_{1},\ldots ,i_{m}=1}^{\infty }\left\vert
U(e_{i_{^{1}}},\ldots ,e_{i_{m}})\right\vert ^{2}\right) ^{\frac{1}{2}}\leq
\left( \int\limits_{0}^{1}\left\vert \sum\limits_{i_{1},\ldots
,i_{m}=1}^{\infty }r_{i_{1}}(t_{1})\cdots
r_{i_{m}}(t_{m})U(e_{i_{^{1}}}.\ldots ,e_{i_{m}})\right\vert ^{2}\right)
^{1/2}.
\end{equation*}%
From the multilinearity of $U$, we can further estimate,%
\begin{eqnarray*}
\left( \sum\limits_{i_{1},\ldots ,i_{m}=1}^{\infty }\left\vert
U(e_{i_{^{1}}},\ldots ,e_{i_{m}})\right\vert ^{2}\right) ^{\frac{1}{2}}
&\leq &\left( \int\limits_{0}^{1}\left\vert U(\sum\limits_{i_{1}=1}^{\infty
}r_{i_{1}}(t_{1})e_{i_{^{1}}},\ldots ,\sum\limits_{i_{m}=1}^{\infty
}r_{i_{m}}(t_{m})e_{i_{m}})\right\vert ^{2}\right) ^{1/2} \\
&\leq &\sup_{t_{1}\cdots ,t_{m}\in \lbrack 0,1]}\left\Vert U\right\Vert
\left\Vert \sum\limits_{i_{1}=1}^{\infty
}r_{i_{1}}(t_{1})e_{i_{^{1}}}\right\Vert \cdots \left\Vert
\sum\limits_{i_{m}=1}^{\infty }r_{i_{m}}(t_{m})e_{i_{^{m}}}\right\Vert  \\
&=&\left\Vert U\right\Vert .
\end{eqnarray*}%
for all continuous $m$--linear forms $U\colon c_{0}\times \cdots \times
c_{0}\rightarrow \mathbb{K}$.

We begin with $m=2$ and consider the standard bilinear form $S_{2}\colon
c_{0}\times c_{0}\rightarrow \mathbb{R}$,
\begin{equation*}
S_{2}(x,y)=x_{1}y_{1}+x_{1}y_{2}+x_{2}y_{1}-x_{2}y_{2}.
\end{equation*}%
As $\left\Vert S_{2}\right\Vert =2$, we conclude that%
\begin{equation*}
C_{(2,1)2}\geq \sqrt{2}.
\end{equation*}

From $m=3$ and on, we start to deform the standard $m$--linear forms in a
non-symmetric fashion. For that, define $S_{3} \colon c_{0}\times
c_{0}\times c_{0}\rightarrow \mathbb{R}$ by
\begin{equation*}
S_{3}(x,y,z)=(z_{1}+z_{2})\left(
x_{1}y_{1}+x_{1}y_{2}+x_{2}y_{1}-x_{2}y_{2}\right) +(z_{1}-z_{2})\left(
x_{3}y_{1}+x_{3}y_{2}+x_{4}y_{1}-x_{4}y_{2}\right).
\end{equation*}%
Note that $\Vert S_{3}\Vert =4$ and, since
\begin{equation*}
\left( \sum\limits_{i_{2},i_{3}=1}^{\infty }\left(
\sum\limits_{i_{1}=1}^{\infty }\left\vert
S_{3}(e_{i_{^{1}}},e_{i_{2}},e_{i_{3}})\right\vert \right) ^{2}\right)
^{1/2}=4\sqrt{4},
\end{equation*}%
we conclude
\begin{equation*}
C_{\left( 2,2,1\right) 3}\geq \left( \sqrt{2}\right) ^{2}.
\end{equation*}%
For $m=4$, we consider $S_{4} \colon c_{0}\times c_{0}\times c_{0}\times
c_{0}\rightarrow \mathbb{R}$ given by
\begin{align*}
& S_{4}(x,y,z,w)= \\
& =\left( w_{1}+w_{2}\right) \left(
\begin{array}{c}
(z_{1}+z_{2})\left( x_{1}y_{1}+x_{1}y_{2}+x_{2}y_{1}-x_{2}y_{2}\right) \\
+(z_{1}-z_{2})\left( x_{3}y_{1}+x_{3}y_{2}+x_{4}y_{1}-x_{4}y_{2}\right)%
\end{array}%
\right) \\
& +\left( w_{1}-w_{2}\right) \left(
\begin{array}{c}
(z_{1}+z_{2})\left( x_{5}y_{1}+x_{5}y_{2}+x_{6}y_{1}-x_{6}y_{2}\right) \\
+(z_{1}-z_{2})\left( x_{7}y_{1}+x_{7}y_{2}+x_{8}y_{1}-x_{8}y_{2}\right).%
\end{array}%
\right)
\end{align*}%
Similarly, we compute $\left\Vert S_{4}\right\Vert =8$ and
\begin{equation*}
\left( \sum\limits_{i_{2},i_{3},i_{4}=1}^{\infty }\left(
\sum\limits_{i_{1}=1}^{\infty }\left\vert
S_{4}(e_{i_{^{1}}},e_{i_{2}},e_{i_{3}},e_{i_{4}})\right\vert \right)
^{2}\right) ^{1/2}=8\sqrt{8}.
\end{equation*}%
Therefore, we obtain
\begin{equation*}
C_{\left( 2,2,2,1\right) 4}\geq \frac{8\sqrt{8}}{8}=\left( \sqrt{2}\right)
^{3}.
\end{equation*}

The construction can be carried out by induction for all $m$--linear form,
through the general non-symmetric procedure:
\begin{eqnarray*}
&&S_{m}\left( x^{(1)},\cdots,x^{(m)}\right) \\
&=&\left( x_{1}^{(m)}+x_{2}^{(m)}\right) S_{m-1}\left(
x^{(1)},\cdots,x^{(m-1)}\right) +\left( x_{1}^{(m)}-x_{2}^{(m)}\right)
S_{m-1}\left( B^{2^{m-1}}\left( x^{(1)}\right)
,x^{(2)},\cdots,x^{(m-1)}\right) ,
\end{eqnarray*}%
where%
\begin{equation*}
B^{2^{m-1}}\left( x^{(1)}\right) =\left(
x_{2^{m-1}+1}^{(1)},x_{2^{m-1}+2}^{(1)},\cdots\right)
\end{equation*}
for all natural number $m$. As before, such a construction yields%
\begin{equation}
C_{\left( 2,2,\cdots,2,1\right) m}\geq \left( \sqrt{2}\right) ^{m-1}.
\label{8765}
\end{equation}

It now follows by H\"{o}lder inequality that if
\begin{equation*}
U\left( x^{(1)},\cdots,x^{(m)}\right) =\sum_{j_{1},\cdots,j_{m}=1}^{\infty
}\alpha _{j_{1}\cdots j_{m}}x_{j_{1}}^{(1)}\cdots x_{j_{m}}^{(m)}
\end{equation*}%
is the sum of exactly $k$ monomials, then
\begin{eqnarray*}
&&\left( \sum_{j_{2},\cdots,j_{m}=1}^{\infty }\left( \sum_{j_{1}=1}^{\infty
}\left\vert U(e_{j_{1}},\cdots,e_{j_{m}})\right\vert \right) ^{2}\right)
^{1/2} \\
&\leq &\left( \sum_{j_{2},\cdots,j_{m}=1}^{\infty }\left( \left(
\sum_{j_{1}\in A_{j_{2}\cdots j_{m}}}^{\infty }\left\vert
U(e_{j_{1}},\cdots,e_{j_{m}})\right\vert ^{2}\right) ^{1/2}\times \left(
\sum_{j_{1}\in A_{j_{2}\cdots j_{m}}}^{\infty }1\right) ^{1/2}\right)
^{2}\right) ^{1/2} \\
&\leq &\left( \sum_{j_{1},\cdots,j_{m}=1}^{\infty }\left\vert
U(e_{j_{1}},\cdots,e_{j_{m}})\right\vert ^{2}\times
\sum_{j_{2},\cdots,j_{m}=1}^{\infty }\left( \sum_{j_{1}\in A_{j_{2}\cdots
j_{m}}}^{\infty }1\right) \right) ^{1/2} \\
&\leq &k^{1/2}\left\Vert U\right\Vert ,
\end{eqnarray*}%
where $A_{j_{2}\cdots j_{m}}=\left\{ j_{1}:\alpha _{j_{1}\cdots j_{m}}\neq
0\right\} $. In the last inequality we have used (\ref{9m}). Thus, if an
extremum $m$--linear form is composed by the sum of exactly $k$ monomials,
we conclude that
\begin{equation*}
k^{1/2}\geq \left( \sqrt{2}\right) ^{m-1},
\end{equation*}%
which means, in terms of entropy, that
\begin{equation}
\ent_{m}(2,2,\cdots ,2,1)\geq 2^{m-1}.  \label{ot}
\end{equation}

Next we recall that, if $1\leq p\leq q,$ then
\begin{equation}
\left( \sum\limits_{i}\left( \sum\limits_{j}\left\vert a_{ij}\right\vert
^{p}\right) ^{\frac{1}{p}q}\right) ^{\frac{1}{q}}\leq \left(
\sum\limits_{j}\left( \sum\limits_{i}\left\vert a_{ij}\right\vert
^{q}\right) ^{\frac{1}{q}p}\right) ^{\frac{1}{p}}.  \label{7j}
\end{equation}%
Thus,
\begin{equation}
C_{\left( 2,2,\cdots , 2,1\right) m}\leq \cdots \leq C_{\left( 1,2,\cdots
,2\right) m}.  \label{8766}
\end{equation}%
Easily one shows that
\begin{equation}
C_{\left( 1,2,\cdots ,2\right) m}\leq \left( \sqrt{2}\right) ^{m-1}
\label{8768}
\end{equation}%
and hence, combining (\ref{8765}), (\ref{8766}) and (\ref{8768}) we reach%
\begin{equation*}
\left( \sqrt{2}\right) ^{m-1}\leq C_{\left( 2,2,\cdots ,2,1\right) m}\leq
\cdots \leq C_{\left( 1,2,\cdots ,2\right) m}\leq \left( \sqrt{2}\right)
^{m-1}.
\end{equation*}
Finally we note that, since each $S_{m}$ is the sum of exactly $4^{m-1}$
monomials, we can estimate
\begin{equation*}
2^{m-1}\leq \ent_{m}(2,2,\cdots, 2,1)\leq 4^{m-1}.
\end{equation*}
Hence, the chain of inequalities (\ref{8766}) and standard symmetry
arguments yield
\begin{equation*}
2^{m-1}\leq \ent_{m}(1,2, \cdots,2),\cdots,\ent_{m}(2,2,\cdots ,2,1)\leq
4^{m-1},
\end{equation*}
and Corollary \ref{21} is finally proved. %\section{Final comments}
\newline

\bigskip We close this section with some additional considerations. The
mixed $\left( \ell _{1},\ell _{2}\right) $-Littlewood inequalities should be
regarded, in a natural way, as extremal Bohnenblust--Hille inequalities, in
the sense that they correspond to multiple exponents with maximum
\textquotedblleft diameter\textquotedblright . More precisely, let us define%
\begin{equation*}
\text{diam}\left( q_{1},\cdots ,q_{m}\right) =\max \left\vert
q_{i}-q_{j}\right\vert .
\end{equation*}%
An exponent $\left( q_{1},\cdots ,q_{m}\right) $ is said to be extremal if $%
\text{diam}\left( q_{1},\cdots ,q_{m}\right) =1,$ and this is the case of
the all mixed $\left( \ell _{1},\ell _{2}\right) $-Littlewood inequalities.

The classical Bohnenblust--Hille exponents, on the other hand, verify $\text{%
diam}\left( q_{1},\cdots ,q_{m}\right) =0.$ Proceeding as in \cite{daniel}
it seems plausible that, when $\text{diam}\left( q_{1},\cdots ,q_{m}\right) $
is close to $1$, optimal constants should grow exponentially. In this
regard, Corollary \ref{21} suggests an interesting parallel between diameter
of the exponent and growth of sharp constants in the Bohnenblust--Hille
inequalities:

\begin{equation*}
\begin{tabular}{|l|l|}
\hline
$\text{Diameter of }\left( q_{1},\cdots,q_{m}\right) $ & $\text{Growth of
the constants}$ \\ \hline
$0\text{ (classical case)}$ & $\text{sublinear }\left( <m^{0.4}\right) $ \\
\hline
$1\text{ (mixed }\left( \ell _{1},\ell _{2}\right) \text{-Littlewood)}$ & $%
\text{Exponential }\left( = \sqrt{2}\right) ^{m-1}.$ \\ \hline
\end{tabular}%
\end{equation*}

\bigskip

The relationship between the growth of the sharp constants and the diameter
of the exponent seems a natural line of investigation. It suggests that the
smaller the diameter, the slower the growth. A quantification of such
implication would shed lights on a number of other issues pertaining to the
theory of Bohnenblust--Hille inequalities.

\section{Estimating the entropy of the classical Bohnenblust--Hille
inequality}

The Hardy--Littlewood inequalities for $m$--linear forms (see \cite{a,
dimant, hardy,pra}) are a sharp generalization of the Bohnenblust--Hille
inequality when we replace $c_{0}$ by $\ell _{p}$. It asserts that for any
integer $m\geq 2$ and $2m\leq p\leq \infty $, there exists a constant $%
C_{m,p}^{\mathbb{K}}\geq 1$ such that,
\begin{equation}
\left( \sum_{j_{1},\cdots,j_{m}=1}^{\infty }\left\vert
T(e_{j_{1}},\cdots,e_{j_{m}})\right\vert ^{\frac{2mp}{mp+p-2m}}\right) ^{%
\frac{mp+p-2m}{2mp}}\leq C_{m,p}^{\mathbb{K}}\left\Vert T\right\Vert ,
\label{i99}
\end{equation}%
for all continuous $m$--linear forms $T\colon \ell _{p}\times \cdots \times
\ell _{p}\rightarrow \mathbb{K}$. The exponent $\frac{2mp}{mp+p-2m}$ is
optimal. Following usual convention in the field, $c_{0}$ is understood as
the proper substitute of $\ell _{\infty }$ when the exponent $p\rightarrow
\infty$. Under such an agreement, one easily checks that taking $p=\infty $
in \eqref{i99} recovers the Bohnenblust--Hille inequality.

Investigations pertaining to the Hardy--Littlewood inequalities are closely
related to phenomena observed in the classical Bohnenblust--Hille
inequality, and the question regarding optimal constants is not different
(see \cite{araujo}). For this reason, and to support future investigations
in the theme, in this section we broaden the definition of entropy to the
domain of Hardy--Littlewood inequalities. The results of this section will
be proved for both settings.

We shall use the notation $\ent_{m}^{HL}\left( \frac{2mp}{mp+p-2m},\cdots ,%
\frac{2mp}{mp+p-2m}\right) $ for the entropy of the Hardy--Littlewood
inequality. Our first lemma, which is of independent interest, is crucial in
this section. It can be understood as a generalization of \cite[Lemma 18.14]%
{diestel} to $\ell _{p}$ spaces, which further sharps the constants to their
optimal values.

\begin{lemma}
\label{lemmaTech} Let $m$ be a positive integer. For all $p>2m,$ we have%
\begin{equation*}
\left( \sum\limits_{i_{1},\ldots ,i_{m}=1}^{\infty }\left\vert
U(e_{i_{^{1}}},\ldots ,e_{i_{m}})\right\vert ^{\frac{2p}{p-2m+2}}\right) ^{%
\frac{p-2m+2}{2p}}\leq \left\Vert U\right\Vert ,
\end{equation*}%
for all continuous $m$--linear forms $U\colon \ell _{p}\times \cdots \times
\ell _{p}\rightarrow \mathbb{K}$.
\end{lemma}

\begin{proof}
We start off by recalling that%
\begin{equation*}
\left( \sum\limits_{i_{1},\ldots ,i_{m}=1}^{\infty }\left\vert
U(e_{i_{^{1}}},\ldots ,e_{i_{m}})\right\vert ^{2}\right) ^{\frac{1}{2}}\leq
\left\Vert U\right\Vert ,
\end{equation*}%
for all continuous $m$--linear forms $U\colon \ell _{2}\times c_{0}\times
\cdots \times c_{0}\rightarrow \mathbb{K}$ (see \cite[Proposition 3.5]{bbb}%
).\ So, since $p>2m,$ it is not difficult to see that
\begin{equation*}
\left( \sum\limits_{i_{1},\ldots ,i_{m}=1}^{\infty }\left\vert
U(e_{i_{^{1}}},\ldots ,e_{i_{m}})\right\vert ^{2}\right) ^{\frac{1}{2}}\leq
\left\Vert U\right\Vert ,
\end{equation*}%
for all continuous $m$--linear forms $U\colon \ell _{p}\times c_{0}\times
\cdots \times c_{0}\rightarrow \mathbb{K}$.

The proof we shall deliver of Lemma \ref{lemmaTech} is based on a technique
that goes back to the paper of Hardy and Littlewood, \cite{hardy}, see also
and \cite{pra}. It consists of analyzing the effect on each one of the $m$
exponents $2$ when we replace $c_{0}$ by $\ell _{p}.$

Let us set $s=\frac{2p}{p-2m+2}$ and $\lambda _{0}=2.$ From the inclusion of
the $\ell _{p}$ spaces and (\ref{9m}) we have%
\begin{equation}
\left( \sum\limits_{j_{i}=1}^{\infty }\left( \sum\limits_{\widehat{j_{i}}%
=1}^{\infty }\left\vert T\left( e_{j_{1}},\cdots ,e_{j_{m}}\right)
\right\vert ^{s}\right) ^{\frac{1}{s}\lambda _{0}}\right) ^{\frac{1}{\lambda
_{0}}}\leq \left\Vert T\right\Vert ,  \label{78}
\end{equation}%
for all $m$--linear forms $T\colon \ell _{p}\times c_{0}\times \cdots \times
c_{0}\rightarrow \mathbb{K}$ and all $i=1,\cdots ,m.$ Hereafter, $%
\sum\limits_{\widehat{j_{i}}=1}^{n}$ means the sum over all $j_{k}$ for all $%
k\neq i.$ Note that $\lambda _{0}<s.$ For for all $j=1,\cdots ,m$ let us set%
\begin{equation*}
\lambda _{j}:=\frac{\lambda _{0}p}{p-\lambda _{0}j}.
\end{equation*}%
It is plain to observe that
\begin{equation*}
\lambda _{k-1}<\lambda _{k}<s
\end{equation*}%
for all $k=1,\cdots ,m-2$ and $\lambda _{m-1}=s.$ In addition, for all $%
j=0,\cdots m-2$, we have
\begin{equation*}
\left( \frac{p}{\lambda _{j}}\right) ^{\ast }=\frac{\lambda _{j+1}}{\lambda
_{j}}.
\end{equation*}%
Now now argue by induction, i.e., assuming for $2\leq k\leq m-1$ that
\begin{equation}
\left( \sum_{j_{i}=1}^{n}\left( \sum_{\widehat{j_{i}}=1}^{n}\left\vert
T_{k-1}(e_{j_{1}},\cdots ,e_{j_{m}})\right\vert ^{s}\right) ^{\frac{1}{s}%
\lambda _{k-1}}\right) ^{\frac{1}{\lambda _{k-1}}}\leq \Vert T_{k-1}\Vert
\label{8880}
\end{equation}%
for all $m$--linear forms $T_{k-1}\colon \underbrace{\ell _{p}\times \cdots
\times \ell _{p}}_{k-1\ \mathrm{times}}\times c_{0}\times \cdots \times
c_{0}\rightarrow \mathbb{K}$, all $i=1,\cdots ,m$ and all positive integers $%
n$, we aim to prove that
\begin{equation*}
\left( \sum_{j_{i}=1}^{n}\left( \sum_{\widehat{j_{i}}=1}^{n}\left\vert
T_{k}(e_{j_{1}},\cdots ,e_{j_{m}})\right\vert ^{s}\right) ^{\frac{1}{s}%
\lambda _{k}}\right) ^{\frac{1}{\lambda _{k}}}\leq \Vert T_{k}\Vert
\end{equation*}%
for all $m$--linear forms $T_{k}\colon \underbrace{\ell _{p}\times \cdots
\times \ell _{p}}_{k\ \mathrm{times}}\times c_{0}\times \cdots \times
c_{0}\rightarrow \mathbb{K}$, all $i=1,\cdots ,m$ and all positive integers $%
n$.

The initial case $k=2$ in (\ref{8880})) is precisely (\ref{78}). Let us then
assume (\ref{8880}) and consider, for $1\leq k\leq m,$ an $m$--linear form $%
T_{k}\in \mathcal{L}(\underbrace{\ell _{p},\cdots ,\ell _{p}}_{k\ \mathrm{%
times}},c_{0},\cdots ,c_{0};\mathbb{K})$ and, for each $x\in B_{\ell _{p}},$
define%
\begin{equation*}
\begin{array}{ccccl}
T_{k}^{(x)} & \colon & \underbrace{\ell _{p}\times \cdots \times \ell _{p}}%
_{k-1\ \mathrm{times}}\times c_{0}\times \cdots \times c_{0} & \rightarrow &
\mathbb{K} \\
&  & (z^{(1)},\cdots ,z^{(m)}) & \mapsto & T_{k}(z^{(1)},\cdots
,z^{(k-1)},xz^{(k)},z^{(k+1)},\cdots ,z^{(m)}),%
\end{array}%
\end{equation*}%
with $xz^{(k)}=(x_{j}z_{j}^{(k)})_{j=1}^{\infty }$. We note that $\Vert
T_{k}\Vert \leq \sup \{\Vert T_{k}^{(x)}\Vert :x\in B_{\ell _{p}}\}$ and by
the induction hypothesis applied to $T_{k}^{(x)}$, we obtain
\begin{equation}
\begin{array}{l}
\displaystyle\left( \sum_{j_{i}=1}^{n}\left( \sum_{\widehat{j_{i}}%
=1}^{n}\left\vert T_{k}\left( e_{j_{1}},\cdots ,e_{j_{m}}\right) \right\vert
^{s}\left\vert x_{j_{k}}\right\vert ^{s}\right) ^{\frac{1}{s}\lambda
_{k-1}}\right) ^{\frac{1}{\lambda _{k-1}}}\vspace{0.2cm} \\
=\displaystyle\left( \sum_{j_{i}=1}^{n}\left( \sum_{\widehat{j_{i}}%
=1}^{n}\left\vert T_{k}\left( e_{j_{1}},\cdots
,e_{j_{k-1}},xe_{j_{k}},e_{j_{k+1}},\cdots ,e_{j_{m}}\right) \right\vert
^{s}\right) ^{\frac{1}{s}\lambda _{k-1}}\right) ^{\frac{1}{\lambda _{k-1}}}%
\vspace{0.2cm} \\
=\displaystyle\left( \sum_{j_{i}=1}^{n}\left( \sum_{\widehat{j_{i}}%
=1}^{n}\left\vert T_{k}^{(x)}\left( e_{j_{1}},\cdots ,e_{j_{m}}\right)
\right\vert ^{s}\right) ^{\frac{1}{s}\lambda _{k-1}}\right) ^{\frac{1}{%
\lambda _{k-1}}}\vspace{0.2cm} \\
\leq \Vert T_{k}^{(x)}\Vert \vspace{0.2cm} \\
\leq \Vert T_{k}\Vert%
\end{array}
\label{guga030}
\end{equation}%
for all $i=1,\cdots ,m$ and all $n.$

Let us first consider the case \noindent\ $i=k.$ Since, as previously
mentioned, $\left( \frac{p}{\lambda _{j-1}}\right) ^{\ast }=\frac{\lambda
_{j}}{\lambda _{j-1}}$, for all $j=1,\cdots ,m$, we have
\begin{equation*}
\begin{array}{l}
\displaystyle\left( \sum_{j_{k}=1}^{n}\left( \sum_{\widehat{j_{k}}%
=1}^{n}\left\vert T_{k}\left( e_{j_{1}},\cdots ,e_{j_{m}}\right) \right\vert
^{s}\right) ^{\frac{1}{s}\lambda _{k}}\right) ^{\frac{1}{\lambda _{k}}}%
\vspace{0.2cm} \\
\displaystyle=\displaystyle\left( \sum_{j_{k}=1}^{n}\left( \sum_{\widehat{%
j_{k}}=1}^{n}\left\vert T_{k}\left( e_{j_{1}},\cdots ,e_{j_{m}}\right)
\right\vert ^{s}\right) ^{\frac{1}{s}\lambda _{k-1}\left( \frac{p}{\lambda
_{k-1}}\right) ^{\ast }}\right) ^{\frac{1}{\lambda _{k-1}}\frac{1}{\left(
\frac{p}{\lambda _{k-1}}\right) ^{\ast }}}\vspace{0.2cm} \\
\displaystyle=\left( \left\Vert \left( \left( \sum_{\widehat{j_{k}}%
=1}^{n}\left\vert T_{k}\left( e_{j_{1}},\cdots ,e_{j_{m}}\right) \right\vert
^{s}\right) ^{\frac{1}{s}\lambda _{k-1}}\right) _{j_{k}=1}^{n}\right\Vert
_{\left( \frac{p}{\lambda _{k-1}}\right) ^{\ast }}\right) ^{\frac{1}{\lambda
_{k-1}}}\vspace{0.2cm} \\
\displaystyle=\left( \sup_{y\in B_{\ell _{\frac{p}{\lambda _{k-1}}%
}}}\sum_{j_{k}=1}^{n}|y_{j_{k}}|\left( \sum_{\widehat{j_{k}}%
=1}^{n}\left\vert T_{k}\left( e_{j_{1}},\cdots ,e_{j_{m}}\right) \right\vert
^{s}\right) ^{\frac{1}{s}\lambda _{k-1}}\right) ^{\frac{1}{\lambda _{k-1}}}%
\vspace{0.2cm} \\
\displaystyle=\left( \sup_{x\in B_{\ell
_{p}}}\sum_{j_{k}=1}^{n}|x_{j_{k}}|^{\lambda _{k-1}}\left( \sum_{\widehat{%
j_{k}}=1}^{n}\left\vert T_{k}\left( e_{j_{1}},\cdots ,e_{j_{m}}\right)
\right\vert ^{s}\right) ^{\frac{1}{s}\lambda _{k-1}}\right) ^{\frac{1}{%
\lambda _{k-1}}}\vspace{0.2cm} \\
\displaystyle=\sup_{x\in B_{\ell _{p}}}\left( \sum_{j_{k}=1}^{n}\left( \sum_{%
\widehat{j_{k}}=1}^{n}\left\vert T_{k}\left( e_{j_{1}},\cdots
,e_{j_{m}}\right) \right\vert ^{s}\left\vert x_{j_{k}}\right\vert
^{s}\right) ^{\frac{1}{s}\lambda _{k-1}}\right) ^{\frac{1}{\lambda _{k-1}}}%
\vspace{0.2cm} \\
\displaystyle\leq \Vert T_{k}\Vert ,%
\end{array}%
\end{equation*}%
where in the last inequality we have used by \eqref{guga030}. Let us now
focus in the remaining cases, namely, when $i\neq k.$ Let $k\in \left\{
1,\cdots ,m-1\right\} $ and, for $i=1,\cdots ,m$ and $n$, denote
\begin{equation*}
S_{i}=\left( \sum_{\widehat{j_{i}}=1}^{n}|T_{k}(e_{j_{1}},\cdots
,e_{j_{m}})|^{s}\right) ^{\frac{1}{s}}.
\end{equation*}%
We then have
\begin{equation*}
\begin{array}{l}
\displaystyle\sum_{j_{i}=1}^{n}\left( \sum_{\widehat{j_{i}}%
=1}^{n}|T_{k}(e_{j_{1}},\cdots ,e_{j_{m}})|^{s}\right) ^{\frac{1}{s}\lambda
_{k}}=\sum_{j_{i}=1}^{n}S_{i}^{\lambda _{k}}=\sum_{j_{i}=1}^{n}\frac{%
S_{i}^{s}}{S_{i}^{s-\lambda _{k}}} \\
\displaystyle=\sum_{j_{i}=1}^{n}\sum_{\widehat{j_{i}}=1}^{n}\frac{%
|T_{k}(e_{j_{1}},\cdots ,e_{j_{m}})|^{s}}{S_{i}^{s-\lambda _{k}}}%
=\sum_{j_{k}=1}^{n}\sum_{\widehat{j_{k}}=1}^{n}\frac{|T_{k}(e_{j_{1}},\cdots
,e_{j_{m}})|^{s}}{S_{i}^{s-\lambda _{k}}}\vspace{0.2cm} \\
\displaystyle=\sum_{j_{k}=1}^{n}\left( \sum_{\widehat{j_{k}}=1}^{n}\frac{%
|T_{k}(e_{j_{1}},\cdots ,e_{j_{m}})|^{\frac{s(s-\lambda _{k})}{s-\lambda
_{k-1}}}}{S_{i}^{s-\lambda _{k}}}|T_{k}(e_{j_{1}},\cdots ,e_{j_{m}})|^{\frac{%
s(\lambda _{k}-\lambda _{k-1})}{s-\lambda _{k-1}}}\right) .%
\end{array}%
\end{equation*}%
Hence, applying H\"{o}lder's inequality with exponents $\left( \frac{%
s-\lambda _{k-1}}{\lambda _{k}-\lambda _{k-1}},\frac{s-\lambda _{k-1}}{%
s-\lambda _{k}}\right) $ and $\left( \frac{\lambda _{k}\left( s-\lambda
_{k-1}\right) }{\lambda _{k-1}\left( s-\lambda _{k}\right) },\frac{\lambda
_{k}\left( s-\lambda _{k-1}\right) }{\left( \lambda _{k}-\lambda
_{k-1}\right) s}\right) $ gives
\begin{align*}
& \displaystyle\sum_{j_{i}=1}^{n}\left( \sum_{\widehat{j_{i}}%
=1}^{n}|T_{k}(e_{j_{1}},\cdots ,e_{j_{m}})|^{s}\right) ^{\frac{1}{s}\lambda
_{k}}\vspace{0.2cm} \\
\displaystyle& \leq \sum_{j_{k}=1}^{n}\left( \left( \sum_{\widehat{j_{k}}%
=1}^{n}\frac{|T_{k}(e_{j_{1}},\cdots ,e_{j_{m}})|^{s}}{S_{i}^{s-\lambda
_{k-1}}}\right) ^{\frac{s-\lambda _{k}}{s-\lambda _{k-1}}}\times \left(
\sum_{\widehat{j_{k}}=1}^{n}|T_{k}(e_{j_{1}},\cdots ,e_{j_{m}})|^{s}\right)
^{\frac{\lambda _{k}-\lambda _{k-1}}{s-\lambda _{k-1}}}\right) \\
\vspace{0.2cm}\displaystyle& \leq \left( \sum_{j_{k}=1}^{n}\left( \sum_{%
\widehat{j_{k}}=1}^{n}\frac{|T_{k}(e_{j_{1}},\cdots ,e_{j_{m}})|^{s}}{%
S_{i}^{s-\lambda _{k-1}}}\right) ^{\frac{\lambda _{k}}{\lambda _{k-1}}%
}\right) ^{\frac{\lambda _{k-1}}{\lambda _{k}}\cdot \frac{s-\lambda _{k}}{%
s-\lambda _{k-1}}}\times \left( \sum_{j_{k}=1}^{n}\left( \sum_{\widehat{j_{k}%
}=1}^{n}|T_{k}(e_{j_{1}},\cdots ,e_{j_{m}})|^{s}\right) ^{\frac{1}{s}\lambda
_{k}}\right) ^{\frac{1}{\lambda _{k}}\cdot \frac{(\lambda _{k}-\lambda
_{k-1})s}{s-\lambda _{k-1}}}.
\end{align*}%
It follows from the estimate already proved in the case $i=k$ that:
\begin{equation}
\displaystyle\left( \sum_{j_{k}=1}^{n}\left( \sum_{\widehat{j_{k}}%
=1}^{n}|T(e_{j_{1}},\cdots ,e_{j_{m}})|^{s}\right) ^{\frac{1}{s}\lambda
_{k}}\right) ^{\frac{1}{\lambda _{k}}\cdot \frac{(\lambda _{k}-\lambda
_{k-1})s}{s-\lambda _{k-1}}}\leq \Vert T\Vert ^{\frac{(\lambda _{k}-\lambda
_{k-1})s}{s-\lambda _{k-1}}},  \label{huhi}
\end{equation}%
Now, applying H\"{o}lder's inequality for $\left( \frac{s}{s-\lambda _{k-1}},%
\frac{s}{\lambda _{k-1}}\right) $ together with \eqref{guga030} yields
\begin{equation}
\begin{array}{l}
\displaystyle\left( \sum_{j_{k}=1}^{n}\left( \sum_{\widehat{j_{k}}=1}^{n}%
\frac{|T_{k}(e_{j_{1}},\cdots ,e_{j_{m}})|^{s}}{S_{i}^{s-\lambda _{k-1}}}%
\right) ^{\frac{\lambda _{k}}{\lambda _{k-1}}}\right) ^{\frac{\lambda _{k-1}%
}{\lambda _{k}}}=\left\Vert \left( \sum_{\widehat{j_{k}}}\frac{%
|T_{k}(e_{j_{1}},\cdots ,e_{j_{m}})|^{s}}{S_{i}^{s-\lambda _{k-1}}}\right)
_{j_{k}=1}^{n}\right\Vert _{\left( \frac{p}{\lambda _{k-1}}\right) ^{\ast }}%
\vspace{0.2cm} \\
\displaystyle=\sup_{y\in B_{\ell _{\frac{p}{\lambda _{k-1}}%
}}}\sum_{j_{k}=1}^{n}|y_{j_{k}}|\sum_{\widehat{j_{k}}=1}^{n}\frac{%
|T_{k}(e_{j_{1}},\cdots ,e_{j_{m}})|^{s}}{S_{i}^{s-\lambda _{k-1}}}%
=\sup_{x\in B_{\ell _{p}^{n}}}\sum_{j_{k}=1}^{n}\sum_{\widehat{j_{k}}=1}^{n}%
\frac{|T_{k}(e_{j_{1}},\cdots ,e_{j_{m}})|^{s}}{S_{i}^{s-\lambda _{k-1}}}%
|x_{j_{k}}|^{\lambda _{k-1}}\vspace{0.2cm} \\
\displaystyle=\sup_{x\in B_{\ell _{p}}}\sum_{j_{i}=1}^{n}\sum_{\widehat{j_{i}%
}=1}^{n}\frac{|T_{k}(e_{j_{1}},\cdots ,e_{j_{m}})|^{s-\lambda _{k-1}}}{%
S_{i}^{s-\lambda _{k-1}}}|T_{k}(e_{j_{1}},\cdots ,e_{j_{m}})|^{\lambda
_{k-1}}|x_{j_{k}}|^{\lambda _{k-1}}\vspace{0.2cm} \\
\displaystyle\leq \sup_{x\in B_{\ell _{p}}}\sum_{j_{i}=1}^{n}\left( \left(
\sum_{\widehat{j_{i}}=1}^{n}\frac{|T_{k}(e_{j_{1}},\cdots ,e_{j_{m}})|^{s}}{%
S_{i}^{s}}\right) ^{\frac{s-\lambda _{k-1}}{s}}\times \left( \sum_{\widehat{%
j_{i}}=1}^{n}|T_{k}(e_{j_{1}},\cdots ,e_{j_{m}})|^{s}|x_{j_{k}}|^{s}\right)
^{\frac{1}{s}\lambda _{k-1}}\vspace{0.2cm}\right) \\
\displaystyle=\sup_{x\in B_{\ell _{p}}}\sum_{j_{i}=1}^{n}\left( \sum_{%
\widehat{j_{i}}=1}^{n}|T_{k}(e_{j_{1}},\cdots
,e_{j_{m}})|^{s}|x_{j_{k}}|^{s}\right) ^{\frac{1}{s}\lambda _{k-1}}\leq
\Vert T_{k}\Vert ^{\lambda _{k-1}}.%
\end{array}
\label{huho}
\end{equation}%
Combining \eqref{huhi} and \eqref{huho} gives
\begin{equation*}
\displaystyle\sum_{j_{i}=1}^{n}\left( \sum_{\widehat{j_{i}}%
=1}^{n}|T_{k}(e_{j_{1}},\cdots ,e_{j_{m}})|^{s}\right) ^{\frac{1}{s}\lambda
_{k}}\leq \Vert T_{k}\Vert ^{\lambda _{k-1}\frac{s-\lambda _{k}}{s-\lambda
_{k-1}}}\Vert T_{k}\Vert ^{\frac{(\lambda _{k}-\lambda _{k-1})s}{s-\lambda
_{k-1}}}=\Vert T_{k}\Vert ^{\lambda _{k}}.
\end{equation*}%
It finally remains to verify the case $k=m-1$. Since $\lambda _{m-1}=s,$
applying the case $i=k$ we can estimate%
\begin{equation*}
\left( \sum_{j_{i}=1}^{n}\left( \sum_{\widehat{j_{i}}%
=1}^{n}|T_{k}(e_{j_{1}},\cdots ,e_{j_{m}})|^{s}\right) ^{\frac{1}{s}\lambda
_{m}}\right) ^{\frac{1}{s}}=\left( \sum_{j_{m}=1}^{n}\left( \sum_{\widehat{%
j_{m}}=1}^{n}|T_{k}(e_{j_{1}},\cdots ,e_{j_{m}})|^{s}\right) ^{\frac{1}{s}%
\lambda _{m}}\right) ^{\frac{1}{s}}\leq \Vert T_{k}\Vert ,
\end{equation*}%
and the proof of the Lemma is complete.
\end{proof}

We can now establish the following lower bound estimate for the entropy in
the Hardy-Littlewood inequality. {First, we recall a result from \cite%
{campos}, which will be used herein as technical lemma and we sketch its
proof for the sake of completeness:}

\begin{lemma}
\label{4444} Let $m\geq 2$ and $p\geq 2m$. The optimal constants of the
Hardy-Littlewood inequalities satisfies
\begin{equation*}
C_{\mathbb{R},m,p}\geq \frac{2^{\frac{2mp+2m-p-2m^{2}}{mp}}}{\sup_{x\in
\lbrack 0,1]}\frac{((1+x)^{p^{\ast }}+(1-x)^{p^{\ast }})^{\frac{1}{p^{\ast }}%
}}{{(1+x^{p})^{1/p}}}}.
\end{equation*}
\end{lemma}

\begin{proof}
(Sketch) Consider the natural isometric isomorphism $\Psi \colon \mathcal{L}%
\left( ^{2}\ell _{p}^{2};\mathbb{R}\right) \rightarrow \mathcal{L}\left(
\ell _{p}^{2};\left( \ell _{p}^{2}\right) ^{\ast }\right) .$ For
\begin{equation*}
\begin{array}{ccccl}
T_{2,p} & \colon & \ell _{p}^{2}\times \ell _{p}^{2} & \rightarrow & \mathbb{%
R} \\
&  & ((x_{i}^{(1)}),(x_{i}^{(2)})) & \mapsto &
x_{1}^{(1)}x_{1}^{(2)}+x_{1}^{(1)}x_{2}^{(2)}+x_{2}^{(1)}x_{1}^{(2)}-x_{2}^{(1)}x_{2}^{(2)},%
\end{array}%
\end{equation*}%
we have
\begin{equation*}
\begin{array}{ccccl}
\Psi (T_{2,p}) & \colon & \ell _{p}^{2} & \rightarrow & \ell _{p^{\ast }}^{2}
\\
&  & (x_{i}) & \mapsto & (x_{1}+x_{2},x_{1}-x_{2}).%
\end{array}%
\end{equation*}%
Therefore%
\begin{equation*}
\Vert T_{2,p}\Vert =\left\Vert \Psi (T_{2,p})\right\Vert =\sup_{x\in \lbrack
0,1]}\frac{((1+x)^{p^{\ast }}+(1-x)^{p^{\ast }})^{\frac{1}{p^{\ast }}}}{%
(1+x^{p})^{1/p}},
\end{equation*}%
where the norm of the linear operator $\Psi (T_{2,p})$ is performed by using
the best constants from the Clarkson's inequality in the real case (see \cite%
[Theorem 2.1]{mali}). Defining inductively
\begin{equation*}
\begin{array}{rccl}
T_{m,p} \colon & \ell _{p}^{2^{m-1}}\times \cdots \times \ell _{p}^{2^{m-1}}
& \rightarrow & \mathbb{R} \\
& (x^{(1)},\cdots,x^{(m)}) & \mapsto &
(x_{1}^{(m)}+x_{2}^{(m)})T_{m-1,p}(x^{(1)},\cdots,x^{(m)}) \\
&  &  & +(x_{1}^{(m)}-x_{2}^{(m)})T_{m-1,p}(B^{2^{m-1}}(x^{(1)}),%
\cdots,B^{2}(x^{(m-1)})),%
\end{array}%
\end{equation*}%
where $x^{(k)}=(x_{j}^{(k)})_{j=1}^{{2^{m-1}}}\in \ell _{p}^{2^{m-1}}$, $%
1\leq k\leq m$, and $B$ is the backward shift operator in $\ell
_{p}^{2^{m-1}}$, we have
\begin{align*}
|T_{m,p}(x^{(1)},\cdots,x^{(m)})|& \leq
|x_{1}^{(m)}+x_{2}^{(m)}||T_{m-1,p}(x^{(1)},\cdots,x^{(m)})| \\
&
+|x_{1}^{(m)}-x_{2}^{(m)}||T_{m-1,p}(B^{2^{m-1}}(x^{(1)}),B^{2^{m-2}}(x^{(2)}),\cdots,B^{2}(x^{(m-1)}))|
\\
& \leq \Vert T_{m-1,p}\Vert
(|x_{1}^{(m)}+x_{2}^{(m)}|+|x_{1}^{(m)}-x_{2}^{(m)}|) \\
& \leq 2\Vert T_{m-1,p}\Vert \Vert x^{(m)}\Vert _{p},
\end{align*}%
i.e.,
\begin{equation*}
\Vert T_{m,p}\Vert \leq 2^{m-2}\sup_{x\in \lbrack 0,1]}\frac{((1+x)^{p^{\ast
}}+(1-x)^{p^{\ast }})^{\frac{1}{p^{\ast }}}}{(1+x^{p})^{1/p}}
\end{equation*}%
and hence
\begin{equation*}
C_{\mathbb{R},m,p}\geq \frac{(4^{m-1})^{\frac{mp+p-2m}{2mp}}}{2^{m-2}\Vert
T_{2,p}\Vert }=\frac{2^{\frac{2mp+2m-p-2m^{2}}{mp}}}{\sup_{x\in \lbrack 0,1]}%
\frac{\left( \left( 1+x\right) ^{p^{\ast }}+\left( 1-x\right) ^{p^{\ast
}}\right) ^{1/p^{\ast }}}{\left( 1+x^{p}\right) ^{1/p}}}.
\end{equation*}
\end{proof}

\begin{theorem}
\label{999}For $m\geq 2$ and $p>2m,$ there holds
\begin{equation*}
\ent_{m}^{HL}\left( \frac{2mp}{mp+p-2m},\cdots ,\frac{2mp}{mp+p-2m}\right)
\geq 2^{\frac{2\left( p-2m\right) \left( m-1\right) }{2m^{2}+p-4m}}.
\end{equation*}
\end{theorem}

\begin{proof}
Our starting point is the thesis of previous Lemma which assures that
\begin{equation*}
\left( \sum\limits_{i_{1},\ldots ,i_{m}=1}^{\infty }\left\vert
U(e_{i_{^{1}}},\ldots ,e_{i_{m}})\right\vert ^{\frac{2p}{p-2m+2}}\right) ^{%
\frac{p-2m+2}{2p}}\leq \left\Vert U\right\Vert ,
\end{equation*}%
for all continuous $m$--linear forms $U\colon \ell _{p}\times \cdots \times
\ell _{p}\rightarrow \mathbb{R}$. By H\"{o}lder inequality, if an extremum $U
$ is composed by the sum of exactly $k$ monomials, we have

\begin{eqnarray*}
&&\left( \sum\limits_{i_{1},\ldots ,i_{m}=1}^{\infty }\left\vert
U(e_{i_{^{1}}},\ldots ,e_{i_{m}})\right\vert ^{\frac{2mp}{mp+p-2m}}\right) ^{%
\frac{mp+p-2m}{2mp}} \\
&\leq &\left( \left( \sum\limits_{i_{1},\ldots ,i_{m}=1}^{\infty }\left\vert
U(e_{i_{^{1}}},\ldots ,e_{i_{m}})\right\vert ^{\frac{2mp}{mp+p-2m}\times
\frac{p-2m+mp}{m\left( p-2m+2\right) }}\right) ^{\frac{m\left( p-2m+2\right)
}{p-2m+mp}}\left( \sum\limits_{i_{1},\ldots ,i_{m}=1}^{\infty }1^{\frac{%
p-2m+mp}{2m^{2}-4m+p}}\right) ^{\frac{2m^{2}-4m+p}{p-2m+mp}}\right) ^{\frac{%
mp+p-2m}{2mp}} \\
&\leq &\left\Vert U\right\Vert k^{\frac{2m^{2}-4m+p}{2mp}}.
\end{eqnarray*}

By Lemma \ref{4444} we have
\begin{equation*}
C_{m,p}\geq \frac{2^{\frac{2mp+2m-p-2m^{2}}{mp}}}{\sup_{x\in \lbrack 0,1]}%
\frac{\left( \left( 1+x\right) ^{p^{\ast }}+\left( 1-x\right) ^{p^{\ast
}}\right) ^{1/p^{\ast }}}{\left( 1+x^{p}\right) ^{1/p}}}\geq \frac{2^{\frac{%
2mp+2m-p-2m^{2}}{mp}}}{2}.
\end{equation*}%
Last estimate finally yields%
\begin{equation*}
k^{\frac{2m^{2}-4m+p}{2mp}}\geq 2^{\frac{2mp+2m-p-2m^{2}}{mp}-1},
\end{equation*}%
and hence%
\begin{equation*}
k\geq 2^{\frac{2\left( p-2m\right) \left( m-1\right) }{2m^{2}+p-4m}},
\end{equation*}%
which concludes the proof of the current theorem.
\end{proof}

\begin{corollary}
\label{66} For $m\geq 2,$
\begin{equation*}
\ent_{m}\left( \frac{2m}{m+1},\cdots ,\frac{2m}{m+1}\right) \geq 4^{m-1}.
\end{equation*}
\end{corollary}

\bigskip The following two corollaries are immediate consequences of the
proof of Theorem \ref{999}:

\begin{corollary}
\label{7y}Let $m\geq 2$ be an integer and $K\geq 1$ be a real number. If we
are restricted to $m$--linear forms with up to $K^{m}$ monomials, then the
optimal constants of the $m$--linear Bohnenblust--Hille inequalities are
bounded by $\sqrt{K}$.
\end{corollary}

\begin{corollary}
If we are restricted to $m$--linear forms with up to $4^{m-1}$ monomials,
the optimal constants of the $m$--linear Bohnenblust--Hille inequalities are
precisely $2^{1-\frac{1}{m}}$.
\end{corollary}

\bigskip

\begin{remark}
Initially, we note that in Lemma \ref{lemmaTech}, we have $\frac{2p}{p-2m+2}%
=2$ if and only if $p=\infty $. Interestingly enough, it is not possible to
attain the exponent $2$ when $p<\infty .$ In fact, if a universal estimate
as
\begin{equation}
\left( \sum_{j_{1},\cdots ,j_{m}=1}^{\infty }\left\vert T(e_{j_{1}},\cdots
,e_{j_{m}})\right\vert ^{\eta }\right) ^{\frac{1}{\eta }}\leq \left\Vert
T\right\Vert ,  \label{98p}
\end{equation}%
holds for some $\eta \geq 1$, then by plugging the $m$--linear from Lemma %
\ref{4444} into (\ref{98p}), one reaches
\begin{equation*}
1\geq \frac{\left( \sum_{j_{1},\cdots ,j_{m}=1}^{\infty }\left\vert
T(e_{j_{1}},\cdots ,e_{j_{m}})\right\vert ^{\eta }\right) ^{\frac{1}{\eta }}%
}{\left\Vert T\right\Vert }\geq \frac{\left( 4^{m-1}\right) ^{\frac{1}{\eta }%
}}{2^{m-2}\sup_{x\in \lbrack 0,1]}\frac{((1+x)^{p^{\ast }}+(1-x)^{p^{\ast
}})^{\frac{1}{p^{\ast }}}}{{(1+x^{p})^{1/p}}}},
\end{equation*}%
and thus%
\begin{equation*}
\eta \geq \frac{2\left( m-1\right) }{\log _{2}\left( \sup_{x\in \lbrack 0,1]}%
\frac{((1+x)^{p^{\ast }}+(1-x)^{p^{\ast }})^{\frac{1}{p^{\ast }}}}{{%
(1+x^{p})^{1/p}}}\right) +\left( m-2\right) }>2.
\end{equation*}
\end{remark}

\section{Why does interpolation seem not to be the optimal approach?\label{g}%
}

In this final section, we discuss the eventual (and, in our opinion, quite
likely) impossibility of finding sharp constants in Bohnenblust--Hille
inequalities by means of interpolation techniques, which ultimately suggests
that new technologies will be required to fully access such an optimality
problem.

The results established so far in this work gather evidences to support the
conjecture that the optimal constants in the Bohnenblust--Hille inequality
are uniformly bounded in $m$; they are likely to be precisely $2^{1-\frac{1}{%
m}}.$ We have noted that for $m=2$ (the case $m=1$ is obvious), the entropy
of the classical exponent of the Bohnenblust--Hille inequality coincides
with the entropy of the exponents of the mixed $\left( \ell _{1},\ell
_{2}\right) $--Littlewood inequalities. The entropy of the exponents of any
mixed $\left( \ell _{1},\ell _{2}\right) $--Littlewood inequalities grow
exponentially with $m$. Restricting to $m$-forms composed by the sum of $%
4^{m-1}$ monomials (and this is the number of monomials needed to attain the
optimal constants of any mixed $\left( \ell _{1},\ell _{2}\right) $%
--Littlewood inequalities) the optimal constants of the Bohnenblust--Hille
inequalities are $2^{1-\frac{1}{m}}.$ In this final section, we offer a
further technical argument which enforces such evidences, at the expenses,
however, of the eminence of interpolation techniques.

For the sake of fairness, we open the floor by emphasizing that
interpolation methods have promoted significant advances in the theory of
Bohnenblust--Hille inequalities and several recent developments indicate
that eventually they were a definitive (optimal) approach. By way of
example, we mention the radical improved of the original exponential upper
bounds to sublinear growth. That is, a celebrated result established by
means of interpolation methods, in \cite{bayart}, assures the existence of a
constant $\kappa $ such that
\begin{equation}
B_{\mathbb{R},m}<\kappa m^{\frac{2-\log 2-\gamma }{2}},  \label{87j}
\end{equation}%
where $\gamma $ is the Euler-Mascheroni constant. Hence, it gives a
sublinear growth as $m^{\frac{2-\log 2-\gamma }{2}}\approx m^{0.36482}$.
Nonetheless, despite several attempts, the best known lower bounds for $B_{%
\mathbb{R},m}$ are still $2^{1-\frac{1}{m}}.$ In what follows we will
establish an abstract formula for the upper bounds for $B_{\mathbb{K},m}$,
and hereafter in this section we consider both the real and complex fields,
which aim to call the attention on why the interpolation procedure used in
the proof of (\ref{87j}) is potentially not optimal.

Let is revisit the interpolative procedure used in \cite{a, bayart} in the
illustrative case $m=3$.%
\begin{eqnarray*}
&&\left( \sum_{j_{1},\cdots ,j_{3}=1}^{\infty }\left\vert
U(e_{j_{1}},e_{j_{2}},e_{j_{3}})\right\vert ^{\frac{3}{2}}\right) ^{2/3} \\
&\leq &\left[ \left( \sum_{j_{1},j_{2}=1}^{\infty }\left(
\sum_{j_{3}=1}^{\infty }\left\vert
U(e_{j_{1}},e_{j_{2}},e_{j_{3}})\right\vert ^{1}\right) ^{\frac{1}{1}\times
2}\right) ^{1/2}\right] ^{1/3} \\
&&\times \left[ \left( \sum_{j_{1}=1}^{\infty }\left( \sum_{j_{2}=1}^{\infty
}\left( \sum_{j_{3}=1}^{\infty }\left\vert
U(e_{j_{1}},e_{j_{2}},e_{j_{3}})\right\vert ^{2}\right) ^{\frac{1}{2}\times
1}\right) ^{\frac{1}{1}\times 2}\right) ^{\frac{1}{2}}\right] ^{1/3} \\
\times  &&\left[ \left( \sum_{j_{1}=1}^{\infty }\left(
\sum_{j_{2},j_{3}=1}^{\infty }\left\vert
U(e_{j_{1}},e_{j_{2}},e_{j_{3}})\right\vert ^{2}\right) ^{\frac{1}{2}\times
1}\right) ^{1/1}\right] ^{1/3}.
\end{eqnarray*}%
Hence, the optimal constant $B_{3}$ for the $3$--linear Bohnenblust--Hille
inequality satisfies%
\begin{equation}
B_{3}\leq \left( \sup_{ijk}\left\{
C_{(1,2,2)3}C_{(2,1,2)3}C_{(2,2,1)3}\right\} \right) ^{1/3}.  \label{sw}
\end{equation}%
If we use the multiple exponents $(\frac{4}{3},\frac{4}{3},2),(\frac{4}{3},2,%
\frac{4}{3})$ and $(2,\frac{4}{3},\frac{4}{3})$ we further obtain
\begin{equation}
B_{3}\leq \left( \sup_{ijk}\left\{ C_{(\frac{4}{3},\frac{4}{3},2)3}C_{(2,%
\frac{4}{3},\frac{4}{3})3}C_{(\frac{4}{3},2,\frac{4}{3})3}\right\} \right)
^{1/3},  \label{ou}
\end{equation}%
where the notation $\sup_{ijk}$ means that the supremum is taken over the
product of the constants, considering the same $3$--linear form when
estimating the constants and the order of the sums is always kept the same.
This is, by all possible means, a rather uniform {information, and the
estimate} (\ref{ou}) is, at least, as good as the best known estimate for
the constants of the $3$--linear Bohnenblust--Hille inequality, i.e., $%
2^{3/4}$.

It has been observed in \cite{a} that the interpolation of the mixed $\left(
\ell _{1},\ell _{2}\right) $--Littlewood inequalities is not an optimal
choice. In fact, using interpolation in this setting one simply gets $\left(
\sqrt{2}\right) ^{m-1}$ as the upper bounds for the $m$--linear
Bohnenblust--Hille constants, and this is far from optimal. On the other
hand, in \cite{bayart} it is noted that the interpolation of%
\begin{equation*}
\left( \frac{2\left( m-1\right) }{\left( m-1\right) +1},\cdots ,\frac{%
2\left( m-1\right) }{\left( m-1\right) +1},2\right) ,\cdots ,\left( 2,\frac{%
2\left( m-1\right) }{\left( m-1\right) +1},\cdots ,\frac{2\left( m-1\right)
}{\left( m-1\right) +1}\right)
\end{equation*}%
is much more effective, which ultimately gives sublinear upper bounds as (%
\ref{87j}). In what follows, however, we will remark that, insofar as
optimality of the constants are concerned, this is yet not the most
favorable procedure. Rather, the key point seems to rely on keeping the
uniformity when carrying interpolation procedures: the same $m$--linear form
be considered in the whole process, which obviously yields a potentially sub
optimal constant.

In a broad sense, this approach shows that the optimal Bohnenblust--Hille
constants for $m$--linear forms is%
\begin{equation}
B_{m}\leq \inf \left\{
\begin{array}{c}
\sup_{i_{1},\cdots,i_{m}}\left\{ \left( C_{(1,2,\cdots,2)m}\cdots
C_{(2,2,\cdots,2,1)m}\right) ^{1/m}\right\} ,\cdots, \\
\sup_{i_{1},\cdots,i_{m}}\left\{ \left( C_{\left( \frac{2\left( m-1\right) }{%
\left( m-1\right) +1},\cdots,\frac{2\left( m-1\right) }{\left( m-1\right) +1}%
,2\right) m}\cdots C_{\left( \frac{2\left( m-1\right) }{\left( m-1\right) +1}%
,\cdots,\frac{2\left( m-1\right) }{\left( m-1\right) +1},2\right) m}\right)
^{1/m}\right\}%
\end{array}%
\right\} ,  \label{om}
\end{equation}%
where, again, the supremum is taken over the product of the constants,
considering the same $m$--linear form when estimating the constants and,
moreover, the order of the sums is always the same. It turns out that this
new strategy is rather sensible to the $m$--linear form chosen to estimate
the optimal constants of the mixed $\left( \ell _{1},\ell _{2}\right) $%
--Littlewood inequalities, as performed in Section \ref{8nn}. In general we
should not expect that the same $m$--linear form optimizes all the extrema $%
m $--linear forms used in the interpolation process and this is why the
interpolation process seems to be sub-optimal. For reasons of symmetry, we
can choose the natural order of the sums $\sum\limits_{i_{1}}\cdots
\sum\limits_{i_{m}}$ in (\ref{om}).

It is obvious that the estimates given in (\ref{om}) are at least as good as
the best known estimates given in \cite{bayart}. While, it still seems hard
to give a definite step forward as to improve the upper estimates of the $m$%
--linear Bohnenblust--Hille constants, we believe that some computational
assisted simulations could accomplish some {advances. In fact, we conjecture
that the optimal constants for the $m$--linear Bohnenblust--Hille for real
scalars are precisely $2^{1-\frac{1}{m}}.$}

The effective computation of the above formulas seem to be somewhat puzzling
in view of the uniformity (the same multilinear form has to be chosen, and
the same order of the sums), and it is our belief that this is the reason
why the previous proofs of the Bohnenblust--Hille inequality could not
provide better constants.

The notion of universally extremum $m$--linear form is strongly connected to
what we have discussed in this section so far. We will say that a continuous
$m$--linear form $T_{ext}\colon c_{0}\times \cdots \times c_{0}\rightarrow
\mathbb{K}$ is universally extremum for the mixed $\left( \ell _{1},\ell
_{2}\right) $--Littlewood inequalities if
\begin{equation*}
\left\{
\begin{array}{c}
\sum_{j_{1}=1}^{\infty }\left( \sum_{j_{2},\cdots,j_{m}=1}^{\infty
}\left\vert T_{ext}(e_{j_{1}},\cdots,e_{j_{m}})\right\vert ^{2}\right) ^{%
\frac{1}{2}}\leq \left( \sqrt{2}\right) ^{m-1}\left\Vert U\right\Vert , \\
\vdots \\
\left( \sum_{j_{1},\cdots,j_{m-1}=1}^{\infty }\left( \sum_{j_{m}=1}^{\infty
}\left\vert T_{ext}(e_{j_{1}},\cdots,e_{j_{m}})\right\vert \right)
^{2}\right) ^{1/2}\leq \left( \sqrt{2}\right) ^{m-1}\left\Vert U\right\Vert .%
\end{array}%
\right.
\end{equation*}%
Note that the order of the sums in all $m$ inequalities {is} preserved. We
note that our approach to reach the optimal constants of the mixed $\left(
\ell _{1},\ell _{2}\right) $--Littlewood inequalities have shown that
different inequalities potentially need different extrema multilinear forms
and, moreover, the order of the sums is rather important for the
effectiveness of the computations. We are led to believe, henceforth, that
it is quite unlikely that there exists an universally extremum $m$--linear
form $T_{ext}$ for $m>2$ as we have defined. Of course the same definition
of universally extremum multilinear forms will be stated for the other
interpolation sequences
\begin{equation*}
\left( \left( \frac{2k}{k+1},\cdots,\frac{2k}{k+1},2,\cdots,2\right)
,\cdots,\left( 2,\cdots,2,\frac{2k}{k+1},\cdots,\frac{2k}{k+1}\right)
\right) .
\end{equation*}

After so many attention devoted to discuss the Bohnenblust--Hille inequality
we think that it worths to present, with a quite simple proof, a formally
stronger version of (\ref{u88}); the interesting case is the classical one,
so this is just a formally stronger result:

\begin{theorem}
\label{rec}Let $m$ be a positive integer. There is a constant $D_{m}^{%
\mathbb{K}}\geq 1$ such that%
\begin{equation}
\left( \sum\limits_{i_{1},\cdots,i_{m}=1}^{\infty }\left\vert
\prod\limits_{s=1}^{m}T_{s}\left( e_{i_{1}},\cdots,e_{i_{m}}\right)
\right\vert ^{\frac{2}{m+1}}\right) ^{\frac{m+1}{2m}}\leq D_{m}^{\mathbb{K}%
}\prod\limits_{s=1}^{m}\left\Vert T_{s}\right\Vert ^{\frac{1}{m}}
\label{u5r}
\end{equation}%
for all continuous $m$--linear forms $T_{s} \colon c_{0}\times \cdots \times
c_{0}\rightarrow \mathbb{K}$ and all $s=1,\cdots,m.$
\end{theorem}

In \eqref{u5r}, when $T_{1}=\cdots =T_{m}$, we recover the classical
Bohnenblust--Hille inequality; the proof keeps for $D_{m}^{\mathbb{K}}$ the
best known upper bounds for (\ref{u88}), i.e., there are constants $\kappa
_{1},\kappa _{2}$ such that
\begin{equation}
B_{\mathbb{R},m}<\kappa _{1}m^{\frac{2-\log 2-\gamma }{2}},  \label{km}
\end{equation}%
\begin{equation}
B_{\mathbb{C},m}<\kappa _{2}m^{\frac{1-\gamma }{2}}  \label{kb}
\end{equation}%
where $\gamma $ is the Euler-Mascheroni constant ($m^{\frac{2-\log 2-\gamma
}{2}}\approx m^{0.36482}$ and $m^{\frac{1-\gamma }{2}}\approx m^{0.21139}$).

\bigskip \textbf{Proof of Theorem \ref{rec}.}

Since%
\begin{equation*}
\frac{1}{\frac{2}{m+1}}=\frac{1}{\frac{2\left( m-1\right) }{\left(
m-1\right) +1}}+\overset{m-1\text{ times}}{\cdots }+\frac{1}{\frac{2\left(
m-1\right) }{\left( m-1\right) +1}}+\frac{1}{2},
\end{equation*}%
we can apply H\"{o}lder inequality for mixed exponents (see \cite[page 50]%
{fournier} or \cite{bene}) as to reach%
\begin{eqnarray*}
&&\left( \sum\limits_{i_{1},\cdots,i_{m}=1}^{\infty }\left\vert
\prod\limits_{s=1}^{m}T_{s}\left( e_{i_{1}},\cdots,e_{i_{m}}\right)
\right\vert ^{\frac{2}{m+1}}\right) ^{\frac{m+1}{2m}} \\
&=&\left[ \left( \sum\limits_{i_{1},\cdots,i_{m}=1}^{\infty }\left\vert
\prod\limits_{s=1}^{m}T_{s}\left( e_{i_{1}},\cdots,e_{i_{m}}\right)
\right\vert ^{\frac{2}{m+1}}\right) ^{\frac{m+1}{2}}\right] ^{\frac{1}{m}} \\
&\leq &\left(
\begin{array}{c}
\left[ \sum\limits_{i_{1}=1}^{\infty }\left(
\sum\limits_{i_{2},..,i_{m}=1}^{\infty }\left\vert T_{1}\left(
e_{i_{1}},\cdots,e_{i_{m}}\right) \right\vert ^{\frac{2\left( m-1\right) }{%
\left( m-1\right) +1}}\right) ^{\frac{\left( m-1\right) +1}{2\left(
m-1\right) }\times 2}\right] ^{\frac{1}{2}}\times \cdots \times \\
\times \cdots \times \left[ \sum\limits_{i_{1},\cdots,i_{m-1}=1}^{\infty
}\left( \sum\limits_{i_{m}=1}^{\infty }\left\vert T_{m}\left(
e_{i_{1}},\cdots,e_{i_{m}}\right) \right\vert ^{2}\right) ^{\frac{1}{2}%
\times \frac{2\left( m-1\right) }{\left( m-1\right) +1}}\right] ^{\frac{%
\left( m-1\right) +1}{2\left( m-1\right) }}%
\end{array}%
\right) ^{\frac{1}{m}}.
\end{eqnarray*}

In view of the Khinchin inequality and the optimal constants for the case of
Rademacher functions for real scalars and Steinhaus functions for complex
scalars (see \cite{haa, ko}) as in (\cite{bayart}) and (\ref{7j}), each one
of the $m$ products above are dominated by
\begin{equation*}
\left( \prod\limits_{j=2}^{m}\Gamma \left( 2-\frac{1}{j}\right) ^{\frac{j}{%
2-2j}}\right) \left\Vert T_{s}\right\Vert
\end{equation*}%
for complex scalars and%
\begin{equation*}
\left( 2^{\frac{446381}{55440}-\frac{m}{2}}\prod\limits_{j=14}^{m}\left(
\frac{\Gamma \left( \frac{3}{2}-\frac{1}{j}\right) }{\sqrt{\pi }}\right) ^{%
\frac{j}{2-2j}}\right) \left\Vert T_{s}\right\Vert
\end{equation*}%
for real scalars. The proof now follows as in (\cite{bayart}); the above
estimates yield (\ref{km}) and (\ref{kb}).

\bigskip We conclude the current manuscript with a list of guiding questions
which we believe are relevant to advance the understanding in the theory of
Bohnenblust--Hille inequalities.

\begin{problem}
\label{1}Is the asymptotic growth of $\ent_{m}\left( \frac{2m}{m+1},\cdots,%
\frac{2m}{m+1}\right) $ in fact exponential?
\end{problem}

\begin{problem}
Is there some intermediate (between the classical and extremals) exponent $%
\left( q_{1},\cdots,q_{m}\right) $ such that the optimal constants
associated have a polynomial growth? In other words, is there $\delta \in
(0,1)$ and $\left( q_{1},\cdots,q_{m}\right) $\ with%
\begin{equation*}
\mathrm{diam}\left( q_{1},\cdots,q_{m}\right) =\delta
\end{equation*}%
such that the optimal growth of the constants associated to $\left(
q_{1},\cdots,q_{m}\right) $ is polynomial?
\end{problem}

\begin{problem}
\label{2}Is it true that for $m>2$ there are no universally extremum $m$%
--linear forms for the mixed $\left( \ell _{1},\ell _{2}\right) $%
--Littlewood inequalities?
\end{problem}

\begin{problem}
\label{3}Is it true that for $m>2$ there are no universally extremum $m$%
--linear forms for
\begin{equation*}
\left( \frac{2\left( m-1\right) }{\left( m-1\right) +1},\cdots,\frac{2\left(
m-1\right) }{\left( m-1\right) +1},2\right) ,\cdots,\left( 2,\frac{2\left(
m-1\right) }{\left( m-1\right) +1},\cdots,\frac{2\left( m-1\right) }{\left(
m-1\right) +1}\right) ?
\end{equation*}
\end{problem}

A positive solution to Problem \ref{1} would imply that the sequence of
optimal constants of the $m$--linear Bohnenblust--Hille inequality is
bounded. Positive solutions to Problems \ref{2} and/or \ref{3} are
essentially a confirmation that the interpolation procedure is not optimal.

\subsection{Final comments}

It is evident that the Khinchin inequality is a key result in the proof of
the Bohnenblust--Hille inequality. As we have remarked in Section \ref{8zzz}%
, for $1\leq p<p_{0}$ the optimal constants are $A_{p}=2^{\frac{1}{p}-\frac{1%
}{2}}.$ It is quite simple to check that these estimates are attained when $%
\left( a_{j}\right) _{j=1}^{\infty }=\alpha \left( \pm e_{i}\pm e_{j}\right)
$ for some $\alpha \neq 0$ and $i\neq j.$ When $p=1$ a decisive fact used in
Section \ref{8zzz} is that the optimal estimates of the Khinchin inequality
are achieved if and only if $\left( a_{j}\right) _{j=1}^{\infty }=\alpha
\left( \pm e_{i}\pm e_{j}\right) $ (Szarek's result, see also \cite{De} for
a robust Khinchin Inequality with a qualitative approach to Szarek's result)$%
.$ This was one of the main ingredients to prove Proposition \ref{7vg}.
Although we have not found in the literature, we believe that Szarek's
result (and eventually a qualitative version of it) also holds for $%
p<p_{0}\cong 1.8474$ (we are using the notation of Section \ref{8zzz}).

\begin{conjecture}
\label{conj}For $p<p_{0}$ the optimal constants of the Khinchin inequality
are attained if and only if $\left( a_{j}\right) _{j=1}^{\infty }=\alpha
\left( \pm e_{i}\pm e_{j}\right) $ for some $\alpha \neq 0$ and $i\neq j.$
\end{conjecture}

Up to now, the best upper bound for $B_{\mathbb{R},3}$ is $2^{3/4}.$
Following the lines of the proof of Proposition \ref{7vg}, and supposing the
above conjecture verified, we can prove that if $2^{3/4}$ is in fact sharp
and
\begin{equation*}
T(x,y,z)=\sum\limits_{i,j,k=1}^{\infty }a_{ijk}x_{i}y_{j}z_{k}
\end{equation*}%
is an extremum, then%
\begin{equation*}
\comp\left( T\right) =2.
\end{equation*}%
Furthermore, a condition analogous to (\ref{8h2}) and (\ref{8h3}) is
verified in this context. As we have mentioned before, we believe that $%
2^{3/4}$ is not sharp and there are not extrema as described above. Similar
considerations hold for $3<m\leq 13$ because $\frac{2\left( m-1\right) }{%
\left( m-1\right) +1}<p_{0}$ and the optimal constants of the Khinchin
inequality still behave as in the case $m=2.$ We stress that if we show that the in fact $%
2^{3/4}$ is not the optimal constant for $m=3$, it is immediate that the other best known values for $m>3$ are also not optimal.


\bigskip

\begin{thebibliography}{99}
\bibitem{fournier} R. Adams, J.J. Fournier, Sobolev spaces, Elsevier, 2003.

\bibitem{a} N. Albuquerque, F. Bayart, D. Pellegrino, J. Seoane-Sep\'{u}%
lveda, Sharp generalizations of the multilinear Bohnenblust--Hille
inequality, J. Funct. Anal. \textbf{266} (2014), 3726--3740.

\bibitem{araujo} G. Araujo, D. Pellegrino, D. Diniz P. da Silva e Silva, On
the upper bounds for the constants of the Hardy--Littlewood inequality, J.
Funct. Anal. \textbf{267} (2014), 1878--1888.

\bibitem{bayart} F. Bayart, D. Pellegrino, J. Seoane-Sepulveda, The Bohr
radius of the $n$-dimensional polydisk is equivalent to $\sqrt{\frac{\log n}{%
n}},$ Adv. Math. \textbf{264} (2014), 726--746.

\bibitem{boas} H.P. Boas, The football player and the infinite series,
Notices Amer. Math. Soc. \textbf{44} (1997), 1430--1435.

\bibitem{bene} A. Benedek, R. Panzone, The space $L^{p}$, with mixed norm,
Duke Math. J. \textbf{28} (1961), 301--324.

\bibitem{bh} H. F. Bohnenblust and E. Hille, On the absolute convergence of
Dirichlet series, Ann. of Math. \textbf{32} (1931), 600--622.

\bibitem{bbb} G. Botelho, D.\ Pellegrino, When every multilinear mapping is
multiple summing. Math. Nachr. \textbf{282} (2009), no. 10, 1414--1422.

\bibitem{campos} J. Campos, W. Cavalcante, V. F\'{a}varo, D. Nu\~{n}ez-Alarc%
\'{o}n, D. Pellegrino, D.M. Serrano-Rodr\'{\i}guez , Polynomial and
multilinear Hardy--Littlewood inequalities: analytical and numerical
approaches, arXiv:1503.00618.

\bibitem{ccc} W. Cavalcante, D. Pellegrino, Geometry of the closed unit ball
of the space of bilinear forms on $\ell _{p}$, \qquad arXiv:1603.01535
[math.FA].

\bibitem{De} A. De, I. Diakonikolas, R.A. Servedio, A robust Khintchine
inequality, and algorithms for computing optimal constants in Fourier
analysis and high-dimensional geometry. SIAM J. Discrete Math. 30 (2016),
no. 2, 1058--1094.

\bibitem{defseip} A. Defant, L. Frerick, J. Ortega-Cerd\`{a}, M. Ouna\"{\i}%
es, K. Seip, The Bohnenblust-Hille inequality for homogeneous polynomials is
hypercontractive. Ann. of Math. (2) \textbf{174} (2011), no. 1, 485--497.

\bibitem{diestel} J. Diestel, H. Jarchow, A. Tonge, Absolutely summing
operators, cambridge stusies in advances mathematics 43, 2005.

\bibitem{dimant} V. Dimant and P. Sevilla--Peris, Summation of coefficients
of polynomials on $\ell _{p}$ spaces, arXiv:1309.6063v1 [math.FA], to appear
in Publ. Math.

\bibitem{haa} U. Haagerup, The best constants in the Khintchine inequality,
Studia Math. \textbf{70} (1981), 231--283.

\bibitem{hardy} G. Hardy and J. E. Littlewood, Bilinear forms bounded in
space $[p,q]$, Quart. J. Math. \textbf{5} (1934), 241--254.

\bibitem{ko} H. K\"{o}nig, On the best constants in the Khintchine
inequality for Steinhaus variables, Israel Math. Journal \textbf{203}
(2014), 23--57

\bibitem{LLL} J. E. Littlewood, On bounded bilinear forms in an infinite
number of variables, Quart. J. (Oxford Ser.) \textbf{1} (1930), 164--174.

\bibitem{mali} L. Maligranda, N. Sabourova, On Clarkson's inequality in the
real case. Math. Nachr, \textbf{280} (2007), 1363--1375.

\bibitem{montanaro} A. Montanaro, Some applications of hypercontractive
inequalities in quantum information theory. J. Math. Phys. \textbf{53}
(2012), no. 12, 122206, 15 pp.

\bibitem{daniel} D. Pellegrino, The optimal constants of the mixed $(\ell
_{1},\ell _{2})$--Littlewood inequality, J. Number Theory \textbf{160}
(2016), 11--18.

\bibitem{dia} D. Pellegrino and D.M. Serrano-Rodr\'{\i}guez, On the mixed $%
(\ell _{1},\ell _{2})$--Littlewood inequality for real scalars and
applications, arXiv:1510.00909v1 [math.FA].

\bibitem{pra} T. Praciano--Pereira, On bounded multilinear forms on a class
of $\ell _{p}$ spaces, J. Math. Anal. Appl. \textbf{81} (1981), 561--568.

\bibitem{seip} E. Saksman, K. Seip, Some open questions in analysis for
Dirichlet series, To appear in the proceedings volume for the conference
"Completeness Problems, Carleson Measures, and Spaces of Analytic Functions"
held at the Mittag--Leffler Institute in 2015.

\bibitem{szarek} J. Szarek, On the best constants in the Khinchin
inequality. Studia Math. \textbf{58} (1976), no. 2, 197--208.
\end{thebibliography}
\end{document}